\documentclass{irmaart}

\usepackage{epsfig,graphics,color}

\theoremstyle{plain}

\newtheorem{theorem}{Theorem}[section]

\newtheorem{corollary}[theorem]{Corollary}

\newtheorem{lemma}[theorem]{Lemma}

\newtheorem{proposition}[theorem]{Proposition}

\theoremstyle{definition}

\newtheorem{definition}[theorem]{Definition}

\newtheorem{remark}[theorem]{Remark}
\numberwithin{equation}{section}

\newcommand{\Hom}{\operatorname{Hom}}

\newcommand{\cC}{\mathcal{C}}

\newcommand{\cF}{\mathcal{F}}

\newcommand{\cW}{\mathcal{W}}

\renewcommand{\L}{{\mathbb{L}}}

\newcommand{\R}{{\mathbb{R}}}

\newcommand{\Z}{{\mathbb{Z}}}

\newcommand{\kk}{{\mathbf{k}}}

\newcommand{\id}{\operatorname{id}}

\newcommand{\im}{\operatorname{im}}

\newcommand{\End}{\operatorname{End}}

\newcounter{llistadepth}

\newenvironment{manlist}[1]{\addtocounter{llistadepth}{1}
      \edef\llistacontador{llista\romannumeral\the\value{llistadepth}}
      \list{({#1{\llistacontador}})}{\usecounter{\llistacontador}
      \def\makelabel##1{\hss\llap{##1}}
      \itemsep=2pt\parsep=0pt\topsep=3pt plus 1pt minus 1 pt}}{\endlist
      \addtocounter{llistadepth}{-1}}

\newenvironment{romlist}{\begin{manlist}{\roman}}{\end{manlist}}
\newenvironment{numlist}{\begin{manlist}{\arabic}}{\end{manlist}}

\newcommand{\Ext}{\mathrm{Ext}}
\newcommand{\Tor}{\mathrm{Tor}}

\newcommand{\RHom}{\mathrm{RHom}}

\newcommand{\ad}{\mathrm{ad}}
\newcommand{\Ho}{\mathrm{Ho}}

\usepackage{xypic}

\renewcommand{\c}{{\mathbf{C}}}
\renewcommand{\d}{{\mathbf{D}}}
\newcommand{\ch}{{\mathbf{CH}}}

\author{Hossein Abbaspour}

\begin{document}

\address{Laboratoire Jean Leray, Universit\'e de Nantes,\\ 2, rue de la Houssini\`ere  \\ Nantes 44300, France\\
email: {\tt abbaspour@univ-nantes.fr}} 

\title{On algebraic structures of the Hochschild complex}

\maketitle

%\abstract{We first review various known algebraic structures on the Hochschild (co)homology of a differential graded algebras under weak Poincar\'e duality  hypothesis, such as Calabi-Yau algebras,  derived Poincar\'e duality algebras and closed Frobenius algebras. This includes a BV-algebra structure on $HH^*(A,A^\vee)$ or $HH^*(A,A)$, which in the latter case is an extension of the natural Gerstenhaber structure on $HH^*(A,A)$.  As an example, after proving that the chain complex of the Moore loop space of a manifold $M$ is a CY-algebra and using Burghelea-Fiedorowicz-Goodwillie theorem we obtain a BV-structure on the homology of the free space.
%
% In Section 6 we prove that these BV/coBVstructures can be  indeed defined for the Hochschild homology of a symmetric open Frobenius DG-algebras. In particular we prove that the Hochschild homology and cohomology of a symmetric open Frobenius algebra is a BV and coBV-algebra.  In Section 7 we exhibit a BV structure on the shifted relative Hochschild homology of a symmetric commutative Frobenius algebra. The existence of a BV-structure on the relative  Hochschild homology was expected in the light of Chas-Sullivan and Goresky-Hingston results for free loop spaces. In Section 8 we present an action of Sullivan diagrams on the Hochschild (co)chain complex of a closed Frobenius DG-algebra. This recovers Tradler-Zeinalian~\cite{TZ} result for closed Froebenius algebras using the isomorphism $C^*(A ,A) \simeq C^*(A,A^\vee)$.}
 
\tableofcontents

\section*{Introduction}
In this article we study the algebraic structures of Hochschild homology and cohomology of differential graded associative algebras over a field $\kk$ in four settings: Calabi-Yau algebras, derived Poincar\'e duality algebras, open Frobenius algebras and closed Frobenius algebras. For instance, we prove the existence of a Batalin-Vilkovisky (BV) algebra structure on the Hochschild cohomology $HH^*(A,A)$ in the first two cases, and  on the Hochschild cohomology $HH^*(A,A^\vee)$ in the last two cases. 

Let us  explain the main motivation of the results presented in this chapter.
One knows from Chen~\cite{Chen} and Jones~\cite{Jones}  that the homology of $LM=C^\infty(S^1,M)$, the free loop space of a simply connected manifold $M$, can be computed by
\begin{equation}\label{eqgiso}
H_*(LM)\simeq HH^*(A,A^\vee),
\end{equation}
where $A=C^*(M)$ is the singular cochain algebra of $M$. Jones also proved an equivariant version $H_*^{S^1}(LM)\simeq HC^*(A)$. Starting with the work of Chas and Sullivan, many different algebraic structures on $H_*(LM)$ have been discovered. This includes a BV-algebra structure on $H_*(LM)$~\cite{CS1}, and an action of Sullivan chord diagrams on $H_*(LM)$ which in particular implies that $H_*(LM)$ is an open Frobenius algebra. In order to find an algebraic model of these structures using the Hochschild complex and the isomorphism above, one has to equip the cochain algebra $A$ with further structures.

In order to find the Chas-Sullivan BV-structure on $HH^*(A,A^\vee)$, one should take into account the Poincar\'e duality for $M$. Over a field $\kk$, we have a quasi-isomorphism $    A \rightarrow  C_*(M) \simeq A^\vee   $ given by capping with the fundamental class of $M$. Therefore one can use the result of Section~\ref{sec:Poincare} to find a BV-algebra structure on  $HH^*(A,A^\vee)\simeq HH^*(A,A)$. Said more explicitly, $H_*(LM)$ is isomorphic to $HH^*(A,A)$ as a BV-algebra, where the underlying Gerstenhaber structure of the BV-structure on $HH^*(A,A)$ is the standard one (see Theorem~\ref{Gersten}). This last statement, which is true over a field, is a result to which many authors have contributed~\cite{CJ, FTBV,Tradler,Mer}. The statement is not proved as yet for integer coefficients.

As we will see in Section~\ref{secmoore}, an alternative way to find an algebraic model for the BV-structure of $H_*(LM)$ is via the Burghelea-Fiedorowicz-Goodwillie isomorphism $H_*(LM)\simeq HH_*(C_*(\Omega M), C_*(\Omega M))$, where $\Omega M$ is the based loop space of $M$. This approach has the advantage of working  for all closed manifolds and it does not require $M$ to be simply connected. Moreover there is not much of a restriction on the coefficients~\cite{Malm}.

Now we turn our attention to the action of the Sullivan chord diagrams and to the open Frobenius algebra structure of $H_*(LM)$ (see~\cite{CG}). For that one has to assume that the cochain algebra has some additional structure. The results of Section~\ref{sec:chord} show that in order to have an action of Sullivan chord diagram  
on $HH^*(A,A^\vee)$ and $HH_*(A,A)$ we have to start with a closed  Frobenius algebra structure on $A$. As far as we know, such structure is not known on $C^*(M)$ but only on the differential forms $\Omega^*(M)$ (see~\cite{LamStan}). Therefore the isomorphism (\ref{eqgiso}) is an isomorphism of algebras over the PROP of Sullivan chord diagrams if we work with real coefficients (see also~\cite{CTZ}).

Here is a brief description of the organization of the chapter. In Section~\ref{section-CH} we introduce the Hochschild homology and cohomology of a differential graded algebra and various classical operations such as cup and cap product. We also give the definition of Gerstenhaber and BV-algebras. In particular we give an explicit description of the Gerstenhaber algebra structure one the Hochschild cohomology of $A$ with coefficients in $A$,  $HH^*(A,A)$. In Section~\ref{sec:derived} we explain how we can see  Hochschild (co)homology as a derived functor. That section includes a quick review  of model categories which can be skipped upon a first reading.

In Sections~\ref{sec:CY} and~\ref{sec:Poincare} we work with algebras satisfying a sort of derived Poincar\'e duality condition rather than being equipped with an inner product. In these two sections, we introduce a BV-structure on $HH^*(A,A)$, whose underlying Gerstenhaber structure is the standard one (see Section~\ref{section-CH}).

In Section~\ref{sec:Frobenius} we show that the Hochschild homology $HH_*(A,A)$ and Hochschild cohomology $HH^*(A,A^\vee)$ of a symmetric open Frobenius algebra $A$ are both BV-algebras. Note that we don't find a BV-algebra structure on $HH^*(A,A)$, although the latter is naturally a Gerstenhaber algebra (see Section~\ref{section-CH}). However, if the open Frobenius  algebra  $A$ is a closed one, then via the induced isomorphism $HH^*(A,A)\simeq HH^*(A,A^\vee)$ we obtain a BV-algebra structure on  $HH^*(A,A)$  whose underlying Gerstenhaber algebra structure is the standard one.  
 
In Section~\ref{sec:Frobenius-shifted} we present a BV-structure on the relative Hochschild homology $\widetilde{HH}_*(A)$  (after a shift in degree) for a symmetric commutative open Frobenius algebra $A$. We believe that this BV-structure provides a dual algebraic model for the string topology operations on the (relative) homology of the free loop space 
that are  discussed in~\cite{CS2} and~\cite{GorHing}. 
  
 In Section~\ref{sec:chord} we describe an action of the Sullivan chord diagrams on the Hochschild chains of a closed Frobenius algebra $A$, which induces  an action of the homology of the moduli space of curves (Section 2.10 in~\cite{WahWest}). In particular there are BV- and coBV-structures on $HH_*(A,A)$ and on the dual theory $HH^*(A,A^\vee)$. Our construction is inspired by~\cite{TZ} and is equivalent to those given in~\cite{TZ} and~\cite{KS} for closed Frobenius algebras. 
This formulation is very much suitable for an action of the moduli space in the spirit of Costello's construction for Calabi-Yau categories~\cite{CostelloCY}, see also~\cite{WahWest} for this particular case.  Finally, in Section~\ref{sec:cyclic} we will show how a BV-structure on Hochschild (co)homology induces a graded Lie algebra structure, and even better, a gravity algebra structure, on cyclic (co)homology.

\bigskip

\noindent\textbf{Acknowledgements.} I would like to thank Janko Latschev and Alexandru Oancea for encouraging me to write this chapter. I am indebted to Aur\'elien Djament, Alexandre Quesney and Friedrich Wagemann for reading the first draft and suggesting a few corrections. I am also grateful to Nathalie Wahl for pointing out a few errors and missing references in the first version. She also made  a few helpful suggestions on the organization of the content.  

\setcounter{section}{0}

\section{Hochschild complex} \label{section-CH}
Throughout this paper $\kk$ is a field. Let $A=\kk\oplus \bar{A}$ be an augmented unital differential $\kk$-algebra with differential $d_A$ of degree $\deg d_A=+1$, such that $\bar{A}=A/\kk$, or $\bar{A}$ is the kernel of the augmentation $\epsilon: A \rightarrow \kk$.

A \emph{differential graded (DG) $(A,d_A)$-module}, or \emph{$A$-module} for short, is a $\kk$-complex $(M,d_M)$ together with a (left) $A$-module structure
$\cdot: A\times M\rightarrow M$ such that $ d_M(am)=d_A(a)m+(-1)^{|a|}a d_M(m)$. The multiplication map is of degree zero, i.e. $\deg(am)=\deg a+\deg m $. In particular, the identity above implies that the differential $d_M$ has to be of degree $1$.

Similarly, a graded differential \emph{$(A,d)$-bimodule} $(M,d_M)$ is required to satisfy the identity
$$
d_M(amb)=d_A(a)mb+(-1)^{|a|}ad_M (m)b+(-1)^{|a|+|m|}amd_Ab.
$$
Equivalently $M$ is a $(A^e:=A\otimes A^{op}, d_A\otimes 1 +1\otimes d_A)$ DG-module where $A^{op}$ is the algebra whose
underlying graded vector space is $A$ with the opposite multiplication of $A$, i.e. $a\overset{op}{\cdot}b=(-1)^{|a|.|b|}b\cdot a$. 
From now on $Mod(A)$ denotes the category of (left or right) (differential) $A$-modules and  $Mod(A^e)$ denotes the category of differential $A$-bimodules. All modules considered in this article are differential modules. We will also drop the indices from the differential when there is no danger of confusion.

We recall that the \emph{two-sided bar construction}  (\cite{CarEilen,MacLane}) is given by
$B(A,A,A):= A\otimes T(s\bar{A})\otimes A$  equipped with the differential $d=d_{0}+d_{1}$, where
$d_{0}$ is the internal differential for the tensor product complex $A\otimes T(s\bar{A})\otimes A$ given by 
\begin{equation*}
\begin{split}
d_0(a[a_1,\cdots,a_n]b) = & d(a)[a_1,\cdots,a_n]b 
\\
& \qquad - \sum_{i=1}^n (-1)^{\epsilon_{i}}a[a_1,\cdots,d(a_i),\cdots,a_n]b
\\
& \qquad \qquad + (-1)^{\epsilon_{n+1}}Êa[a_1,\cdots,a_n]d(b),
\end{split}
\end{equation*}
and $d_1$ is the external differential given by 
\begin{equation}
\begin{split}
d_{1}(a[a_{1}, \cdots , a_{n}] b)=& (-1)^{|a|}aa_{1}[a_{2}, \cdots , a_{n}] b
\\ 
& \qquad + \sum_{i=2}^{n} (-1)^{\epsilon_{i}}a[a_{1} , \cdots , a_{i-1}a_{i}, \cdots , a_{n}] b
\\ 
&\qquad \qquad -(-1)^{\epsilon_{n}}a[a_{1}, \cdots , a_{n-1}]
a_{n}b.
\end{split}
\end{equation}
  Here $\epsilon_0=|a|$ and $\epsilon_{i}=|a|+|a_{1}|  +\cdots+
|a_{i-1}|-i+1$ for $i\ge 1$. The degree on $B(A,A,A)$ is defined by 
$$
\deg (a[a_{1}, \cdots , a_{n}]b)=|a|+|b|+\sum_{i=1}^n |s(a_i)|=|a|+|b|+ \sum_{i=1}^n |a_i|-n,
$$ 
so that $\deg (d_0+d_1)=+1$. We recall that $sA$ stands for the suspension of $A$, i.e. the shift in degree by $-1$.

We equip $A$ and $A\otimes_\kk A$, or $A\otimes A$ for short, with the \emph{outer} $A$-bimodule structure that is $a(b_1\otimes b_2)c=(ab_1)\otimes (b_2c)$.
Similarly $B(A,A,A)$ is equipped with the outer $A$-bimodule structure. This is a free resolution of $A$ as an $A$-bimodule which allows us to define \emph{Hochschild chains and cochains} of $A$ with
coefficients in $M$.  Then \emph{(normalized) Hochschild chain complex} with coefficients in $M$ is
\begin{equation}
C_*(A,M):= M \otimes_{A^e} B(A,A,A) =M \otimes  T(s \bar{A})
\end{equation}
and comes equipped with a degree +1 differential $ D=d_{0}+d_{1}$.  We recall that $TV=\oplus_{n\geq 0} V^{\otimes n}$ denotes the tensor algebra of a $\kk$-module $V$.

The internal differential is given by
\begin{equation}
\begin{split}
d_{0}(m  [a_{1},\cdots , a_{n}])&= d_M m  [a_{1},\cdots, a_{n}]-\sum_{i=1}^{n} (-1)^{\epsilon_{i}}m  [a_{1},\cdots,d_A a_{i},\cdots, a_{n}]\\
\end{split}
\end{equation}
and the external differential is
\begin{equation}
\begin{split}
d_{1}(m [a_{1}, \cdots , a_{n}] )= & (-1)^{|m|} m a_{1} [a_{2}, \cdots , a_{n}]\\ 
&\qquad + \sum_{i=2}^{n} (-1)^{\epsilon_{i}} m [a_{1},
\cdots, a_{i-1}a_{i}, \cdots ,a_{n}]\\ 
&\qquad \qquad -(-1)^{\epsilon_{n}(|a
_n|+1)} a_{n}m [a_{1}, \cdots ,  a_{n-1}],
\end{split}
\end{equation}
with $\epsilon_0=|m|$ and $\epsilon_{i}=|m|+|a_{1}|  +\cdots |a_{i-1}|-i+1 $ for $i\geq 1$. Note that the degree of  $m [a_{1}, \cdots , a_{n}]$ is $|m|+\sum_{i=1}^{n}|a_i|-n$.

When $M=A$, by definition $(C_*(A), D=d_0+d_1):=(C_*(A,A),D=d_0+d_1)$ is the \emph{Hochschild chain complex of} $A$ and  $HH_*(A,A):=\ker
D/\im D$ is the \emph{Hochschild homology of} $A$.

Similarly we define the $M$-valued \emph{Hochschild cochain complex} of $A$ to be the dual complex $$ C^*(A,M):=\Hom_{A^e} (B(A,A,A),
M)=\Hom_{\kk}(T(s\bar{A}),M). $$ For a homogenous cochain $f\in C^n(A,M)$, the degree $|f|$ is defined to be the degree of the linear
map $f: (s\bar{A})^{\otimes n}\rightarrow M$. In the case of Hochschild cochains, the external differential of $f\in \Hom (s\bar{A}^{\otimes n},
M)$ is
\begin{equation}
\begin{split}
d_{1}(f)(a_{1}, \cdots , a_{n})&=-(-1)^{(|a_1|+1)|f|}a_{1}f(a_{2}, \cdots , a_{n})\\ 
&\quad -\sum_{i=2}^{n}  (-1)^{\epsilon_i} f(a_{1}, \cdots ,
a_{i-1}a_{i}, \cdots a_{n})\\
& \quad +(-1)^{\epsilon_n} f( a_{1}, \cdots , a_{n-1})a_n,
\end{split}
\end{equation}
where $\epsilon_i=|f|+|a_1|+\cdots +|a_{i-1}|-i+1$. The internal differential of $f\in C^*(A,M)$ is
\begin{equation}
\begin{split}
d_{0}f(a_{1}, \cdots , a_{n})&=d_Mf(a_{1}, \cdots , a_{n})-\sum_{i=1}^n (-1)^{\epsilon_i} f(a_{1}, \cdots d_A a_i \cdots , a_{n}).
\end{split}
\end{equation}

\subsubsection*{Gerstenhaber bracket and cup product:}
When $M=A$, for $x\in C^m(A,A)$ and $y\in C^n(A,A)$ one defines the \emph{cup product} $x\cup y\in C^{m+n}(A,A)$ and the \emph{Gerstenhaber
bracket}  $[x,y] \in C^{m+n-1}(A,A)$ by

\begin{equation}
(x\cup y)(a_{1},\cdots , a_{m+n}):=(-1)^{|y|(\sum_{i\leq m} |a_{i}|+1)}x(a_{1},\cdots , a_{m})y(a_{n+1}, \cdots , a_{m+n}),
\end{equation}
 and

 \begin{equation}
 [x,y]:= x\circ y - (-1)^{(|x|+1)(|y|+1)}y\circ x,
 \end{equation}
 where
 \begin{eqnarray*}
 \lefteqn{(x\circ_j y)( a_{1},\cdots ,a_{m+n-1})  = } \\
& = &   (-1)^{(|y|+1)\sum_{i\leq j} (|a_{i}|+1)} x(a_{1}, \cdots , a_{j},  y(a_{j+1}, \cdots , a_{j+m}),\cdots  ).
 \end{eqnarray*}
 and 
 \begin{equation}
 x\circ y=\sum_j x\circ_j y
 \end{equation}
 
Note that this is not an associative product. It turns out that the operations $\cup$ and $[-,-]$ are chain maps, hence they define operations on $HH^*(A,A)$. Moreover, the following holds.

\begin{theorem}\label{Gersten}(Gerstenhaber~\cite{Gers}) $(HH^*(A,A),\cup,[-,-])$  is a \emph{Gerstenhaber algebra} that is:
\begin{enumerate}
    \item $\cup$ induces an associative and graded commutative product,
    \item  $[x,y\cup z] = [x,y]\cup z + (-1)^{(|x|-1)|y|}y\cup[x,z] $ (Leibniz rule),
    \item $[x,y] = -(-1)^{(|x|-1)(|y|-1)} [y,x] $,
    \item $[[x,y],z] = [x,[y,z]] +(-1)^{(|x|-1)(|y|-1)}[y,[x,z]] $ (Jacobi identity).
\end{enumerate}
At chain level the cup product $\cup$ is commutative up to homotopy, with the homotopy being given by $\circ$.
\end{theorem}

In this article we show that under some kind of Poincar\'e duality condition this Gerstenhaber structure is part of a BV-structure.
\begin{definition}\label{BV-def}(Batalin-Vilkovisky algebra) A \emph{BV-algebra} is a Gerstenhaber algebra  $(A^*,\cdot,[-,-])$ with a degree $+1$ operator
$\Delta: A^*\rightarrow A^{*+1}$ whose deviation from being a derivation for the product $\cdot$ is the bracket $[-,-]$, i.e. 
$$ [a,b]:=(-1)^{|a|}\Delta(ab)-(-1)^{|a|}\Delta(a)b-a \Delta(b), $$
and such that $\Delta^2=0$.
\end{definition}

It follows from $\Delta^2=0$ that $\Delta$ is a derivation for the bracket.
In fact the Leibniz identity for $[-,-]$ is equivalent to the \emph{7-term relation}~\cite{Getz}
\begin{equation}\label{7term}
\begin{split}
\Delta(abc)&= \Delta (ab)c +(-1)^{|a|} a\Delta(bc)+ (-1)^{(|a|-1)|b|}b\Delta (ac)\\
  &-\Delta (a) bc - (-1)^{|a|}a \Delta(b) c-(-1)^{|a|+|b|}ab\Delta c.
\end{split}
\end{equation}

Definition~\ref{BV-def} is equivalent to the following one:

\begin{definition}
A \emph{BV-algebra} is a graded commutative associative algebra  $(A^*,\cdot)$ equipped with a degree $+1$ operator $\Delta: A^*\rightarrow A^{*+1}$ which satisfies the 7-term relation (\ref{7term}) and $\Delta^2=0$. (It follows from the 7-term relation that $[a,b]:=(-1)^{|a|}\Delta(ab)-(-1)^{|a|}\Delta(a)b-a \Delta(b)$ is a Gerstenhaber bracket for the graded commutative associative algebra $(A^*,\cdot)$.)
\end{definition}
As we said before the Leibniz identity is equivalent to the 7-term identity and the Jacobi identity follows from $\Delta^2=0$ and the 7-term identity. We refer the reader who is interested in the homotopic aspects of BV-algebras to~\cite{ColVal}.

For $M=A^\vee:=\Hom_\kk(A,\kk)$, by definition 
$$
(C^*(A), D=d_0+d_1):=(C^*(A,A^\vee),d_0+d_1)
$$ 
is the \emph{Hochschild cochain complex of}
$A$, and 
$$
HH^*(A):=\ker D/ \im D
$$ 
is the \emph{Hochschild cohomology} of $A$. It is clear that $C^*(A)$ and $\Hom_k(C_*(A),k)$ are isomorphic as $k$-complexes, therefore the Hochschild cohomology $A$ is the dual theory of the
Hochschild homology of $A$. The Hochschild homology and cohomology of an algebra have an extra feature, which is the existence of the \emph{Connes operators} $B$, respectively $B^\vee$ (\cite{connes}). On the chains we have
\begin{equation}\label{B1}
    B(a_0  [a_1, a_2 \cdots, a_n ])=\sum_{i=1}^{n+1}(-1)^{\epsilon_i} 1[a_{i+1} \cdots a_n,a_0,\cdots,a_i]
\end{equation}
and on the dual theory $C^*(A)=\Hom_\kk (T(s\bar{A}),A^\vee)=\Hom (A \otimes  T(s\bar{A}), \kk)$ we have 
$$ 
(B^\vee \phi)(a_0 [a_1, a_2 \cdots, a_n ])=(-1)^{|\phi|}\sum_{i=1}^{n+1}(-1)^{\epsilon_i} \phi (1[a_{i+1} \cdots a_n,a_0,\cdots,a_i] ), 
$$ 
where $\phi \in C^{n+1}(A)=\Hom(A\otimes (s\bar{A})^{\otimes n+1}  , \kk)$ and $\epsilon_i=(|a_0|+\dots |a_{i-1}|-i)(|a_i|+\dots |a_{n}|-n+i-1)$. In other words $$B^\vee (\phi)= (-1)^{|\phi|} \phi \circ B.$$
Note that $
\deg(B)=-1 \text{ and } \deg B^\vee=+1.
$

\subsubsection*{Warning:} The degree $k$ of a cycle $x\in HH_k(A,M)$ is not given by the number of terms in a tensor product, but by the total degree.

\begin{remark}  In this article we use normalized Hochschild chains and cochains. It turns out that they are quasi-isomorphic to the
non-normalized Hochschild chains and cochains. The proof is the same as the one for algebras given in Proposition~1.6.5 on page 46 of~\cite{Loday}. One only
has to modify the proof to the case of simplicial objects in the category of differential graded algebras. The proof of the Lemma 1.6.6 of~\cite{Loday} works in this setting since the degeneracy maps commute with the internal differential of a simplicial differential graded algebra.

\end{remark}

\subsubsection*{Chain and cochain pairings and noncommutative calculus.}  Here we borrow some definitions and facts from noncommutative calculus
\cite{CTS}. Roughly said, one should think of $HH^*(A,A)$ and $HH_*(A,A)$ respectively as multi-vector fields and differential forms, and of $B$ as the
de Rham differential.

\subsubsection*{1. Contraction or cap product:} The pairing between elements 
$$
a_0 [a_1,\cdots, a_n]\in C_n(A,A),\qquad f\in C^k(A,A), \qquad n\geq k
$$ 
is given by
\begin{eqnarray}\label{contr-1t}
\lefteqn{i_f(a_0 [a_1,\cdots, a_n])}Ê\\
& := & (-1)^{|f|(\sum_{i=1}^k (|a_i|+1) )}a_0 f(a_1,\cdots, a_k)[ a_{k+1},\cdots, a_n]\in C_{n-k}(A,A). \nonumber 
\end{eqnarray}
This is a chain map $C^*(A,A)\otimes C_*(A,A)\rightarrow  C_*(A,A)$ which induces a pairing at cohomology and homology level.

\subsubsection*{2. Lie derivative:} The next operation is the infinitesimal Lie algebra action of $HH^*(A,A)$  on $HH_*(A,A)$ and is given by Cartan's formula
\begin{equation}
L_{f}=[B,i_f].
\end{equation}
Note that the Gerstenhaber bracket on $HH^*(A,A)$ becomes a (graded) Lie bracket after a shift of degrees by one.
This explains also the sign convention below. The three maps $i_f$, $L_f$ and $ B$ form a \emph{calculus}~\cite{CTS}, i.e.
\begin{eqnarray}
L_f&=&[B,i_f]\\
i_{[f,g]}&=&[L_f,i_g]\\
i_{f\cup g}&=& i_f\circ i_g\\
L_{[f,g]}&=&[L_f,L_g]\\
L_{f\cup g}&=&[L_f,i_g]
\end{eqnarray}

As $HH^*(A,A)$ acts on $HH_*(A)=HH_*(A,A)$ by contraction, it also acts on the dual theory $HH^*(A)= HH^*(A,A^\vee)$. More explicitly, $i_f (\phi) \in
C^{n}(A,A^\vee)$ is given by
\begin{equation}\label{contr-2}
i_f(\phi)(a_0[a_1,\cdots, a_n])\!:=(-1)^{|f|(|\phi|+ \sum_{i=0}^k (|a_i|+1) )}\phi (a_0 [f(a_1,\cdots, a_k), a_{k+1},\cdots, a_n]),
\end{equation}
where $\phi\in C^{n-k}(A,A^\vee)$ and $f\in C^k (A,A)$. In other words
$$
i_f(\phi) :=(-1)^{|f||\phi|} \phi\circ i_f.
$$

\section{Derived category of DGA and derived functors} \label{sec:derived}

Now we try to present the Hochschild (co)homology in a more conceptual way, i.e. as a derived functor on the category of $A$-bimodules. We must first introduce an appropriate class of objects which can approximate all $A$-bimodules. This is done properly using the concept of model category introduced by Quillen~\cite{Qui}. It is also the right language for constructing homological invariants of homotopic categories. It will naturally lead us to the construction of  derived categories as well.

\subsection{A quick review of model categories and derived functors} 
The classical references for this subject are  Hovey's book~\cite{Hov} and the Dwyer-Spali\'nsky manuscript~\cite{DwSpa}.
The reader who gets to know the notion of model category for the first time should not worry about the word ``closed'', which now has only a historical bearing. From now on we drop the word ``closed'' from ``closed model category".

\begin{definition} A \emph{model category} is a category $\c$ endowed with three classes of morphisms $\cC$ (\emph{cofibrations}), $\cF$ (\emph{fibrations}) and $\cW$ (\emph{weak equivalences}) such that the following conditions hold:
\begin{itemize}
\item[(MC1)] $\c$ is closed under finite limits and colimits.

\item[(MC2)] Let $f,g\in Mor(\c)$ such that $fg$ is defined. If any two among  $f,g$ and $fg$ are in $\cW$, then the third one is in $\cW$.

\item[(MC3)] Let $f$ be a retract of $g$, meaning that there is a commutative diagram
$$
\xymatrix{A \ar[r]Ê\ar[d]_f & C \ar[r] \ar[d]_g & A \ar[d]_f \\
B \ar[r]Ê& D \ar[r] & B
} 
$$
in which the two horizontal compositions are identities. If $g\in \cC$ (resp. $\cF$ or $\cW)$,  then $f\in  \cC$ (resp. $\cF$ or $\cW)$.

\item[(MC4)] For a commutative diagram as below with $i \in \cC$  and $p\in \cF$, the morphism $f$ making the diagram commutative exists  if
\begin{numlist}
\item $i\in \cW$ (left lifting property (LLP) of fibrations $f\in\cF$ with respect to acyclic cofibrations $i\in \cW\cap \cC$).
\item $p\in \cW$ (right lifting property (RLP) of cofibrations $i\in\cC$ with respect to acyclic fibrations $p\in \cW\cap \cF$).
\end{numlist}
\begin{equation}
\xymatrix{A\ar[d]_-i\ar[r] & X\ar[d]^-p\\ B\ar@{-->}[ru]^-f\ar[r] & Y}
\end{equation}
In the above we call the elements of $ \cW\cap \cC$ (resp.  $\cW\cap \cF$) \emph{acyclic cofibrations} (resp. \emph{acyclic fibrations}).

\item[(MC5)]   Any morphism $f:A\rightarrow B$ can be written as one of the following:
\begin{numlist}
\item $f=pi$ where $p\in \cF$ and $i\in \cC\cap \cW $;
\item $f=pi$ where $p\in \cF\cap \cW $ and $i\in \cC $.
\end{numlist}
\end{itemize}
\end{definition}

In fact, in a model category the lifting properties characterize the fibrations and cofibrations.

\begin{proposition}
In a model category:
\begin{romlist}
\item The cofibrations are the morphisms which have the RLP with respect to
acyclic fibrations.
\item  The acyclic cofibrations are the morphisms which have the RLP with respect
to fibrations.
\item The fibrations are the morphisms which have the LLP with respect to acyclic
cofibrations.
\item The acyclic fibrations in C are the maps which have the LLP with respect
to cofibrations.
\end{romlist}
\end{proposition}

It follows from (MC1) that a model category $\c$ has an initial object $\emptyset$ and a terminal object $\ast$. An object $A\in Obj(\c)$ is called \emph{cofibrant} if the morphism $\emptyset
 \rightarrow A$ is a cofibration and is said to be \emph{fibrant} if the morphism
$A\rightarrow \ast$ is a fibration.

\medskip

\noindent \textbf{Example 1:}\label{ex-CHR} For  any unital associative ring $ R$, let $\ch(R)$ be the category of non-negatively graded chain complexes of left $R$-modules.
The following three classes of morphisms endow $\ch(R)$ with a model category structure:
\begin{numlist}
\item Weak equivalences $\cW$: these are the quasi-isomorphims, i.e. maps of $R$-complexes $f=\{f_k\}_{k\geq 0}:\{M_k\}_{k\in \Z}\rightarrow \{N_k\}_{k\geq 0}$ inducing an isomorphism $f_*: H_*(M)\rightarrow H_*(N)$ in homology.
\item Fibrations $\cF$: $f$ is a fibration if it is (componentwise) surjective, i.e. for all $k\geq 0$, $f_k: M_k\rightarrow N_k$ is surjective.
\item Cofibrations $\cC$: $f=\{f_k\}$ is a cofibration if for all $k\geq 0$, $f_k: M_k\rightarrow N_k$ is injective with a projective $R$-module as its cokernel.
Here we use the standard definition of projective $R$-modules, i.e. modules which are \emph{direct summands of free $R$-modules}.
\end{numlist}

\medskip

\noindent \textbf{Example 2:}
The category \textbf{Top} of topological spaces can be given the structure
of a model category by defining a map $f : X \rightarrow Y$ to be
\begin{romlist}
\item a weak equivalence if $f$ is a homotopy equivalence;
\item a cofibration if $f$ is a Hurewicz cofibration;
\item a fibration if $f$ is a Hurewicz fibration.
\end{romlist}

Let $A$ be a closed subspace of a topological space $B$. We say that the inclusion $i:A\hookrightarrow B$ is a \emph{Hurewicz cofibration} if it has the homotopy extension property that is for all maps $f:B\rightarrow X$, any homotopy $F:A\times [0,1]\rightarrow X$ of $f|_A$ can be extended to a homotopy of $f:B\rightarrow X$.
\begin{equation*}
\xymatrix{B\cup (A\times [0,1]) \ar[r]^-{f\cup F} \ar[d]_-{ id\times {0}\cup (i\times id)}  & X \\ B \times[0,1]\ar@{-->}[ur]& }
\end{equation*}

A \emph{Hurewicz fibration}  is a continuous map $E\rightarrow B$ which has the homotopy lifting property with respect to all continuous maps $X\rightarrow B$, where $X\in$ \textbf{Top}.

\medskip

\noindent \textbf{Example 3:}  The category \textbf{Top} of topological spaces can be given the
structure of a model category by defining $f : X \rightarrow Y$ to be
\begin{romlist}
\item  a weak equivalence when it is a weak homotopy equivalence.
\item a cofibration if it is a retract of a map $X \rightarrow Y'$ in which $Y'$ is obtained
from $X$ by attaching cells,
\item a fibration if it is a Serre fibration.
\end{romlist}

We recall that a \emph{Serre fibration} is a continuous map $E\rightarrow B$ which has the homotopy lifting property with respect to all continuous maps $X\rightarrow B$ where $X$ is a CW-complex (or, equivalently, a cube).

\subsubsection*{Cylinder, path objects and homotopy relation.}  After setting up the general framework, we define the notion of homotopy. A \emph{cylinder object} for $A\in Obj(\c)$ is an object $A\wedge I \in Obj(\c)$ with a \emph{weak equivalence}
$\sim: A\wedge I \rightarrow A$ which factors the natural map  $id_A\sqcup id_A : A\coprod A \rightarrow A$:

$$
id_A\sqcup id_A:A \coprod A \overset{i}{\rightarrow}  A\wedge I \overset{\sim}{\rightarrow}  A
$$

Here  $A\coprod A\in Obj(\c) $ is the colimit, for which one has two structural maps $in_{0},in_{1}:A\to A\coprod A  $.  Let $i_{0}= i\circ in_{0}$ and $i_{1}= i\circ in_{1}$. A cylinder object $A\wedge I$ is said to be \emph{good} if $A \coprod A \rightarrow  A\wedge I$ is a cofibration. By (MC5), every $A \in Obj(\c)$ has a good cylinder object.

\begin{definition}
Two maps $f,g: A\rightarrow B $ are said to be \emph{left homotopic} $ f\overset{l}{\sim} g$ if there is a cylinder object $A\wedge I$  and  $H: A\wedge I\rightarrow B$ such that $f=H\circ i_{0}$ and $g=H\circ i_{1}$. A left homotopy is said to be \emph{good} if the cylinder object $A\wedge I$ is good. It turns out that every left homotopy relation can be realized by a good cylinder object. In addition one can prove that if $B$ is a fibrant object, then a left homotopy for $f$ and $g$ can be refined into a \emph{very good one}, i.e. $A\wedge I\rightarrow A$ is a fibration.

\end{definition}
It is easy to prove the following:

\begin{lemma}\label{equi-rel}
If $A$ is cofibrant, then left homotopy $\overset{l}{\sim}$ is an equivalence relation on $\Hom_{\c}(A,B)$.
\end{lemma}

Similary, we introduce the notion of path objects which will allow us to define right homotopy relation. A \emph{path object} for $A\in Obj(\c)$  is an object $A^{I}\in Obj(\c)$ with a weak equivalence $A\overset{\sim}{\rightarrow} A^I$ and a morphism $p:A^{I}\rightarrow A\times A$
which factors the diagonal map
$$
(id_{A},id_{A}): A\overset{\sim}{\rightarrow} A^{I}\overset{p}{\rightarrow}A\times A.
$$

Let $pr_{0}, pr_{1}: A\times A \rightarrow A$ be the structural projections. Define $p_{i}=pr_{i}\circ p$. A path object $A^ I$  is said to be \emph{good} if $A^{I}\rightarrow  A\times A$ is a fibration. By (MC5) every $A \in Obj(\c)$ has a good path object.

\begin{definition} Two maps $f,g: A\rightarrow B $ are said to be \emph{right homotopic} $ f\overset{r}{\sim} g$ if there is a path object $B^ I$  and  $H: A\rightarrow B^I$ such that $f=p_{0}\circ H$ and $g=p_{1}\circ H$. A right homotopy is said to be \emph{good} if the path object $P^I$ is good. It turns out that every right homotopy relation can be refined into a good one. In addition one can prove that if $B$ is a cofibrant object then a right homotopy for $f$ and $g$ can be refined into a \emph{very good one}, i.e. $B\rightarrow B^I$ is a cofibration.
\end{definition}

 \begin{lemma}\label{equi-rel-r}
If $B$ is fibrant, then right homotopy $\overset{r}{\sim}$ is an equivalence relation on $\Hom_{\c}(A,B)$.
\end{lemma}

One naturally asks  whether being right and left homotopic are related. The following result answers this question.

\begin{lemma}\label{r-l}  Let $f,g: A\rightarrow B$ be two morphisms in a model category $\c$.
\begin{numlist}
\item If $A$ is cofibrant then $f\overset{l}{\sim}g$ implies $f\overset{r}{\sim}g$.

\item If  $B$ is fibrant then $f\overset{r}{\sim}g$ implies $f\overset{l}{\sim}g$.
\end{numlist}
\end{lemma}

\subsubsection*{Cofibrant and Fibrant replacement and homotopy category.}  By applying  (MC5) to the canonical morphism
$\emptyset \rightarrow A$, there is a cofibrant object (not unique) $QA$ and an \emph{acyclic fibration} $p:QA\overset{\sim}{\rightarrow} A $
such that $\emptyset \rightarrow QA\overset{p}{\rightarrow} A$.  If $A$ is cofibrant we can choose $QA=A$.

\begin{lemma}\label{lift}
Given a morphism $f:A\rightarrow  B$ in $\c$, there is a morphism $\tilde{f}:QA\rightarrow QB$ such that the following diagram commutes:
\begin{equation}
\xymatrix{QA\ar[d]^-{p_A} \ar[r]^-{\tilde{f}} & QA \ar[d]^-{p_B}\\ A \ar[r]^-f & B  }
\end{equation}

The morphism $\tilde{f}$ depends on $f$ up to left and right homotopy, and is a weak equivalence if and only $f$ is. Moreover, if $B$ is fibrant then
the right or left homotopy class of $\tilde{f}$  depends only on the left homotopy class of $f$.
\end{lemma}

Similarly one can introduce a fibrant replacement by applying (MC5) to the terminal morphism $A\rightarrow \ast$ and obtain a fibrant object $RA$  with an \emph{acyclic cofibration} $i_A:A\rightarrow RA$.

\begin{lemma}\label{desc}
Given a morphism $f:A\rightarrow  B$ in $\c$, there is a morphism $\tilde{f}:RA\rightarrow RB$ such that the following diagram  commutes:
\begin{equation}
\xymatrix{A\ar[d]^-{i_A} \ar[r]^-{f} & B \ar[d]^-{i_B}\\ RA \ar[r]^-{\tilde{f}} & RB  }
\end{equation}

The morphism $\tilde{f}$ depends on $f$ up to left and right homotopy, and is a weak equivalence if and only $f$ is. Moreover, if $A$ is cofibrant then  right or left homotopy class of $\tilde{f}$  depends only on the right homotopy class of $f$.

\end{lemma}

\begin{remark}\label{remk-fib-cofib}
For a cofibrant object $A$, $RA$ is also cofibrant because the trivial morphism $(\emptyset \rightarrow RA)=(\emptyset \rightarrow A \overset{i_A}{\rightarrow}RA)$ can be written as the composition of two cofibrations, therefore is a cofibration.
In particular, for any object $A$,  $RQA$ is fibrant and cofibrant. Similarly, $QRA$ is  a fibrant and cofibrant object.
\end{remark}

\begin{lemma}\label{w-c-f}Suppose that $f : A \rightarrow X $ is a map in $\c$ between objects $A$ and $X$
which are both fibrant and cofibrant. Then $f$ is a weak equivalence if and only if $f$
has a homotopy inverse, i.e. if and only if there exists a map $g : X \rightarrow A$ such that
the composites $gf$ and $fg$ are homotopic to the respective identity maps.
\end{lemma}

Putting the last three lemmas together, one can make the following definition:

\begin{definition}  The \emph{homotopy category} $\Ho(\c)$ of a model category $\c$  has the same objects as $\c$ and the morphism set $\Hom_{\Ho(\c)}(A,B)$ consists of the (right or left) homotopy classes of  the morphisms in $\Hom_{\c}(RQA,RQB)$. Note that since $RQA$ and $RQB$ are fibrant and cofibrant, the left and right homotopy relations are the same.  There is a natural functor
$H_{\c}: \c\rightarrow \Ho(\c)$ which is the identity on the objects and sends a morphism $f:A\rightarrow B$ to the homotopy class of the morphism
obtained in $\Hom_{\c}(RQ A,RQ B)$ by applying consecutively Lemma~\ref{lift} and Lemma~\ref{desc}.
\end{definition}

\subsubsection*{Localization functor.}  Here we give a brief conceptual description of the homotopy category of a model category. This description relies only on the class of weak equivalences and suggests that weak equivalences encode most of the homotopic properties of the category. Let $W$ be a subset of  the morphisms in a category $\c$. A functor $F:\c\rightarrow \mathbf{D}$ is said to be a \emph{localization of $\c$ with respect to $W$} if the elements of $W$ are sent to isomorphisms and if $F$ is universal for this property, i.e. if $G:\c \rightarrow \mathbf{D'}$ is any another localizing functor then $G$ factors through $F$ via a functor $G':\mathbf{D}\rightarrow \mathbf{D'}$ for which $G'F=G$. It follows from Lemma~\ref{w-c-f} and a little work that:

\begin{theorem} For a model category $\c$, the natural functor $H_{\c}: \c\rightarrow \Ho(\c)$ is a localization of $\c$ with respect to the weak equivalences.
\end{theorem}

\subsubsection*{Derived and total derived functors.}

In this section we introduce the notions of \emph{left derived functor} $LF$  and  \emph{right derived functor} $RF$ of a functor $F:\c\rightarrow \d$ on a model category $\c$. In particular, we spell out sufficient conditions for the existence of $LF$ and $RF$ which provide us a factorization of $F$ via the homotopy categories. If $\d$ happens to be a model category,  then we also introduce the notion of \emph{total derived functor} and provide some sufficient conditions for its existence. 

All functors considered here are covariant, however see Remark~\ref{rem-op}.

\begin{definition} For a functor $F:\c\rightarrow \d$ on a model category $\c$, we consider all pairs $(G,s)$ where $G:\Ho(\c)\rightarrow\d$ is a functor and $s:G H_{\c}\rightarrow F$ is a natural transformation. The \emph{left derived functor} of $F$ is such a pair $(LF,t)$ which is universal from the left, i.e. for any other such pair $(G,s)$ there is a unique natural transformation $t': G\rightarrow LF$ such that $t(t'H_\c):GH_{\c}\rightarrow F$ is $s$.

Similarly one can define the right derived functor $RF: \Ho(\c) \rightarrow \d$ which provides a factorization of $F$ and satisfies the usual universal property from the right. A \emph{right derived functor} of $F$ is a pair $(RF,t)$ where $RF:\Ho(\c)\rightarrow \d$ and $t$ is a natural transformation $t: F \rightarrow RFH_{\c} $ such that for any pair $(G,s)$ there is a unique natural transformation $t':RF\rightarrow G$ such that $(t'H_\c) t:F \rightarrow GH_{\c}$ is $s$.
\end{definition}

The reader can easily check that the derived functors of $F$ are unique up to canonical equivalence.
The following result tells us when do derived functors exist.

\begin{proposition}\label{prop-derived}
\begin{numlist}
\item Suppose that $F:\c\rightarrow \d$ is a functor from a model category $\c$ to a category $\d$, which transforms acyclic cofibrations between cofibrant objects into isomorphims. Then $(LF,t)$, the left derived functor of $F$, exists. Moreover, for any cofibrant object $X$ the map $t_X:LF(X) \rightarrow F(X)$ is an isomorphism.

\item Suppose that $F:\c\rightarrow \d$ is a functor from a model category $\c$ to a category $\d$, which transforms acyclic fibrations between fibrant objects into isomorphisms. Then $(RF,t)$, the right derived functor of $F$, exists. Moreover, for all fibrant object $X$ the map $t_X: RF(X) \rightarrow F(X)  $ is an isomorphism.

\end{numlist}
\end{proposition}

\begin{definition} Let $F:\c\rightarrow \d$ be a functor between two model categories.  The \emph{total left derived functor} $\L F:\Ho(\c)\rightarrow \Ho(\d)$ is the left derived functor of $H_{\d} F: \c\rightarrow \Ho(\d)$. Similarly one defines the \emph{total right derived functor} $\R F:\c\rightarrow \d$ to be the right derived functor of $H_{\d} F: \c\rightarrow \Ho(\d)$.

\end{definition}

\begin{remark}\label{rem-op}
Till now we have defined and discussed the derived functor for covariant functors. We can define the derived functors for contravariant functors as well, for that we only have to work with the opposite category of the source of the functor. A morphism $A\rightarrow B$ in the opposite category is a cofibration (resp. fibration, weak equivalence) if and only if the corresponding morphism $B\rightarrow A$ is a fibration (resp. cofibration, weak equivalence).
\end{remark}

We finish this section with an example.

\medskip

\noindent \textbf{Example 4:} Consider the model category $\ch(R)$ of Example 1 in Section~\ref{sec:derived} and let $M$ be a fixed $R$-module. One defines the functor $F_M:\ch(R)\rightarrow \ch(\Z)$ given by $F_M(N_*)=M\otimes_R N_*$ where $N_*\in \ch(R)$ is a complex of $R$-modules. Let us check that $F=H_{\ch(R)}F_M:\ch(R)\rightarrow \ch(\Z)$ satisfies the conditions of Proposition~\ref{prop-derived}.

Note that in $\ch(R)$ every object is fibrant and a complex $A_*$ is cofibrant if for all $k$, $A_k$ is a projective $R$-module. We have to show that
an acyclic cofibration $f: A_*\rightarrow B_*$ between cofibrant objects $A$ and $B$ is sent by $F$ to an isomorphism. So for all $k$, we have a short exact sequence $0\rightarrow A_*\rightarrow B_*\rightarrow B_*/A_*\rightarrow 0$ where for all $k$, $B_k/A_k$ is also projective.
Since $f$ is a quasi-isomorphism the homology long exact sequence of this short exact sequence tells us that the complex $B_*/A_*$ is acyclic.
The lemma below shows that $ B_*/A_*$ is in fact a projective complex. Therefore we have $B_*\simeq A_*\oplus B_*/A_*$. So $F_M(B_*)\simeq F_M(A_*)\oplus F_M(B_*/A_*)\simeq F_M(A_*)\oplus \bigoplus_n F_M(D(Z_{n-1}(B_*/A_*),n))$. Here $Z_*(X_*):= \ker (d: X_*\to X_{*+1})$ stands for the graded module of the cycles in a given complex $X_*$, and the complex $D(X,n)_*$ is defined as follows: To any $R$-module $X$ and a positive integer $n$, one can associate a complex $\{D(X,n)_k\}_{k\geq 0}$,
\begin{equation*}
D(X,n)_k=\begin{cases} 0, \text{ if } k\neq n,n-1,\\
X, \text{ if } k=n,n-1,
\end{cases}
\end{equation*}
where the only nontrivial differential is the identity map.

It is a direct check that each $F_M(D(Z_{n-1}(B_*/A_*),n))$ is acyclic,
and therefore $H_{\ch(\Z)}(F_M(B))$ is isomorphic to $H_{\ch(\Z)}(F_M(A))$ in the homotopy category $\Ho(\ch(\Z))$.

\begin{lemma} Let $\{C_k\}_{k\geq 0}$ be an acyclic complex where each $C_k$ is a projective $R$-module. Then $\{C_k\}_{k\geq 0}$ is a projective complex, i.e. any level-wise surjective chain complex map $D_*\rightarrow E_*$ can be lifted via any chain complex map $C_*\rightarrow E_*$.

\end{lemma}
\begin{proof}
It is easy to check that if $X$ is a projective $R$-module then $ D_n(X)$ is a projective complex. Let $C_*^{(m)}$ be the complex

\begin{equation*}
C_k^{(m)}=\begin{cases} C_k, \text{ if } k\geq m,\\
Z_k(C), \text{ if } k=m-1,\\
0 \text{ otherwise. }
\end{cases}
\end{equation*}
Here $Z_k(C)$ denotes the space of cycles in $C_k$, and $B_k(C)$ is the space of boundary elements in $C_k$. The acyclicity condition implies that we have an isomorphism $C_*^{(m)}/C_*^{(m+1)}\simeq D(Z_{m-1}(C),m)$. Note that $Z_0(C)=C_0$ is a projective $R$-module and $C_*=C^{(1)}=C^{(2)}\oplus D_1(Z_0(C))$.
Now $D_1(Z_0(C))$ is a projective complex and $C^{(2)}$ also satisfies the assumption of the lemma and vanishes in degree zero. Therefore by applying the same argument one sees that $C^{(2)}=C^{(3)}\oplus D(Z_1(C),2)$. Continuing this process one obtains $C_*=D(Z_0(C),1)\oplus D(Z_1(C),2)\cdots \oplus D(Z_{k-1},k)\oplus\cdots $ where each factor is a projective complex, thus proving the statement.
\end{proof}

We finish this example by computing the left derived functor. For any $R$-module $N$ let $K(N,0)$ be the chain complex concentrated in degree zero where there is a copy of $N$. Since every object is fibrant, a fibrant-cofibrant replacement of $K(N,0)$ is simply a cofibrant replacement. A cofibrant replacement $P_*$ of $K(N,0)$ is exactly a projective resolution (in the usual sense) of $N$ in the category of $R$-modules.  In the homotopy category of $\ch(R)$, $K(N,0)$ and $P$ are isomorphic because by definition
$
\Hom_{Ho(\ch(R))}(K(N,0),P)
$
consists of  the homotopy classes of $\Hom_{\ch(R)}(RQK(N,0),RQP_*)=\Hom_{\ch(R)}(P_*,P_*)$ which contains the identity map. Therefore $\L F(K(N,0))\simeq \L F (P_*)$  and $ \L F (P_*)$ by Proposition~\ref{prop-derived} and the definition of total derived functor is isomorphic to $H_{\ch(R)}F(P_*)=M\otimes_R P_*$.
In particular,
$$
H_* (\L F(K(N,0))= \Tor_*^R(N,M),
$$
where $\Tor^R_*$ is the usual $\Tor_R$ in homological algebra. We usually denote the derived functor $\L F(N)=N\otimes_R^L M$. 
Similarly one can prove that the contravariant functor $N_*\mapsto \Hom_R(N_*, M)$ has a total right derived functor, denoted by $\RHom_R (N_*,M)$, and
$$
H^*(\RHom_R(K(N,0),M))\simeq \Ext^*_R(N,M)
$$
is just the usual $\Ext$ functor (see Remark~\ref{rem-op}).

\bigskip

\subsection{Hinich's theorem and Derived category of DG modules}
The purpose of this section is to introduce a model category and derived functors of DG-modules over a fixed differential graded $\kk$-algebra. From now on we assume that $\kk$ is a field. The main result is essentially due to Hinich~\cite{Hinich}, who introduced a model category structure for algebras over a vast class of operads.

Let $C(\kk)$ be the category of (unbounded) complexes over $\kk$. For $d\in \Z$ let  $M_d\in C(\kk)$ be the complex
$$
\cdots \rightarrow 0\rightarrow \kk=\kk\rightarrow  0\rightarrow 0\cdots
$$
concentrated in degrees $d$ and $d+1$.

\begin{theorem}(Hinich)\label{hinich}  Let $\c$ be a category which admits finite limits and arbitrary colimits and is endowed with
two right and left adjoint functors $(\#,F)$
\begin{equation}
\#: C\rightleftarrows C(\kk):F
\end{equation}
such that for all $A \in Obj(A)$ the canonical map $A\rightarrow A\coprod F(M_d)$ induces a quasi-isomorphism $A^\#
\rightarrow (A
\coprod F(M_d))^{\#}$. Then there is a model category structure on $\c$ where the three distinct classes of morphisms are:
\begin{numlist}
\item Weak equivalences $\cW$: $f\in Mor(\c)$ is in $\cW$ if $f^\#$ is a quasi-isomorphism.
\item Fibrations $\cF$: $f\in Mor(\c)$ is in $\cF$ if $f^\#$ is (componentwise) surjective.
\item Cofibrations $\cC$: $f\in Mor(\c)$ is a cofibration if it satisfies the LLP property with respect to all acyclic fibrations $\cW\cap \cF$.
\end{numlist}
\end{theorem}

As an application of Hinich's theorem, one obtains a model category structure on the category $Mod(A)$ of
(left) differential graded modules over a differential graded algebra $A$. Here $\#$ is the forgetful functor and $F$ is given by tensoring $F(M)= A\otimes_\kk M$.

\begin{corollary}
The category $Mod(A)$ of DG $A$-modules is endowed with a model category structure where
\begin{romlist}
\item weak equivalences are the quasi-isomorphisms.
\item fibrations are level-wise surjections. Therefore all objects are fibrant.
\item cofibrations are the maps that have the left lifting property with respect to all acyclic fibrations.
\end{romlist}

\end{corollary}

In what follows we give a description of cofibrations and cofibrant objects. An excellent reference for this part is~\cite{FT4}.

\begin{definition} An $A$-module $P$ is called a \emph{semi-free extension} of $M$ if $P$ is a union of an increasing family of $A$-modules
$M=P(-1)\subset P(0) \subset\cdots $ where each $P(k)/P(k-1)$ is a free $A$-modules generated by cycles. In particular $P$ is said to be a \emph{semi-free $A$-module} if it is a semi-free extension of the trivial module $0$.
A \emph{semi-free resolution of an $A$-module morphism} $f:M\rightarrow N $ is a semi-free extension $P$ of $M$
with a quasi-isomorphism $P\rightarrow N$ which extends $f$.

In particular a \emph{semi-free resolution of an $A$-module}  $M$ is a semi-free resolution of the trivial map $0\rightarrow M$.
\end{definition}

The notion of a semi-free module can be traced back to~\cite{GugMay}, and~\cite{Drinfeld} is another nice reference for the subject. A $\kk$-complex $(M,d)$ is called a \emph{semi-free complex} if it is semi-free as a differential $\kk$-module. Here $\kk$ is equipped with the trivial differential.  In the case of  a field $\kk$, every positively graded $\kk$-complex is semi-free. It is clear from the definition that a finitely generated semi-free $A$-module is obtained through a finite sequence of extensions of some free $A$-modules of the form $A[n]$, $n\in\Z$.  Here $A[n]$ is $A$ after a shift in degree by $-n$.

\begin{lemma}\label{lemm-filt}Let $M$ be an $A$-module with a filtration $F_0\subset F_1\subset F_2\cdots$ such that $F_0$ and all $F_{i+1}/F_i$ are semifree $A$-modules.   Then $M$ is semifree.
\end{lemma}

\begin{proof}
Since $F_k/F_{k-1}$ is semifree, it has a filtration $   \cdots P^k_l\subset P^k_{l+1}\cdots $ such that $P^k_l/ P^k_{l+1}$  is generated as an $(A,d)$-module by cycles. So one can write  $F_k/F_{k-1}=\oplus_l (A\otimes  Z'_k(l))$ where $Z'_k(l)$ are free (graded) $\kk$-modules such that $d(Z_k(l))\subset \oplus_{j\leq l} Z_k(j)$. Therefore there are free $\kk$-modules  $Z_k(l)$ such that $$F_k=F_{k-1}\bigoplus_{l\geq 0} Z'_k(l)$$ 
and
$$d(Z_k(l))\subset F_{k-1}\bigoplus_{j<l} A\otimes Z_k(j).$$

In particular $M$ is the free $\kk$-module generated by the union of all basis elements $\{z_\alpha\}$  of $Z_k(l)$'s. Now consider the filtration  $P_0\subset P_1 \cdots $ of free $\kk$-modules constructed inductively as follows:  $P_0$ is generates as $\kk$-module by the $z_\alpha$'s which are cycles, i.e.  $dz_\alpha=0$.  Then $P_k$ is generated by those $z_\alpha$'s such that $dz_\alpha \in A\cdot P_{k-1}$. This is clearly a semifree resolution if we prove that $M=\cup_k P_k$. For that, we show by induction on degree that for all $\alpha$, $z_\alpha$ belongs to some $P_k$. Suppose that $z_\alpha \in Z_k(l)$. Then $dz_\alpha \in \oplus A. Z_i(j)$ where $i<k$ or $i=k$ and $j<l$. By the induction hypothesis all $z_\beta$'s in the sum $dz_\alpha$ are in some $P_{m_\beta}$. Therefore $z_\alpha\in P_{m}$ where $m= \max_\beta m_\beta$ and this finishes the proof.
\end{proof}

\begin{remark}\label{rmk-mod}
If we had not assumed that $\kk$ is a field but only a commutative ring then we could still have put a model category structure on $Mod(A)$. This is a special case of the Schwede-Shipley theorem~\cite[Theorem 4.1]{SchShip}.  More details are provided on pages 503-504 of~\cite{SchShip}.
\end{remark}
\begin{proposition} In the model category of $A$-modules, a map$f:M\rightarrow N$  is a cofibration if and only if it is a retract of a semi-free extension $M\hookrightarrow P$. In particular, an $A$-module $M$ is cofibrant if and only if it is a retract of a semi-free $A$-module, i.e. if and only if it is a direct summand of a semi-free $A$-module.
\end{proposition}

Here is a list of properties of semi-free modules which allow us to define the derived functor by means of semi-free resolutions.
\begin{proposition}\label{prop-ess}
\begin{romlist}
\item Any morphism $f: M\rightarrow N$ of $A$-modules has a semi-free resolution. In particular every $A$-module has a semi-free resolution.

\item If $P$ is a semi-free $A$-module,  $\Hom_A(P,-)$ preserves quasi-isomorphisms.

\item Let $P$ and $Q$ be semi-free $A$-modules and $f:P\rightarrow Q$ be a quasi-isomorphism. Then
$$
g\otimes f  :M\otimes _A P\rightarrow N\otimes_A Q
$$
is a quasi-isomorphism if $g:M\rightarrow N$ is a quasi-isomorphism.
\item Let $P$ and $Q$ be semi-free $A$-modules and $f:P\rightarrow Q$ be a quasi-isomorphism. Then
$$
\Hom_R(g, f)  :\Hom_A(Q,M)\rightarrow \Hom_A(P,N)
$$
is a quasi-isomorphism if $g:M\rightarrow N$ is a quasi-isomorphism.
\end{romlist}
\end{proposition}

The second statement in proposition~\ref{prop-ess} implies that a quasi-isomorphism  $f:M\rightarrow N$ between semi-free $A$-modules is a homotopy equivalence, i.e. there is a map $f':N\rightarrow M$ such that $ ff'-id_N= [d_N,h']$ and $f'f-id_M=[d_M,h]$ for some $h:M\rightarrow N$  and $h':N\rightarrow M$. In fact part (iii) and (iv) follow easily from this observation.

The properties listed above  imply that the functors $-\otimes _AM $ and $\Hom_A(-,M)$ preserve enough weak equivalences, ensuring that the derived functors $\otimes^L_A $ and $\RHom_A(-,M)$ exist for all $A$-modules $M$.

Since we are interested  in Hochschild and cyclic (co) homology, we switch to the category of DG $A$-bimodules.  This category is the same as  the category of DG $A^e$-modules. Therefore one can endow  $A$-bimodules with a model category structure and define the derived functors  $-\otimes_{A^e} ^L M$ and $\RHom_{A^e} (-,M)$ by means of fibrant-cofibrant replacements.

More precisely, for two $A$-bimodules  $M$ and $N$ we have
$$
\Tor_*^{A^{e}}(M,N)= H_*(P\otimes_{A^e} N)
$$
and
$$
\Ext^*_{A^{e}}(M,N)=  H^*(\Hom_{A^{e}}(P, N))
$$
where $P$ is cofibrant replacement for $M$.

By Proposition~\ref{prop-ess} every $A^e$-module has a semi-free resolution. There is an explicit construction of the latter using the
two-sided bar construction. For right and left $A$-modules $P$ and $M$, let
\begin{equation}
B(P,A,M)=\bigoplus_{k\geq 0} P\otimes (s\bar{A})^{\otimes k} \otimes M
\end{equation}
equipped with the following differential:
\begin{itemize}
\item if $k=0$,
\begin{equation*}
D(p[\quad] m)= dp [\quad] n+(-1)^{|p|} p[\quad] dm
\end{equation*}
\item if $ k>0$
\begin{eqnarray*}
\lefteqn{D(  p[a_{1},\cdots , a_{k}]m)}\\
& = & d_0( p[a_{1},\cdots , a_{k}]m)+ d_1( p[a_{1},\cdots , a_{k}]m) \\
& = & dp[a_{1},\cdots , a_{k}]m-\sum_{i=1}^{k} (-1)^{\epsilon_{i}}p  [a_{1},\cdots,d a_{i},\dots a_{k}]m \\
& & \qquad \qquad \qquad \qquad \qquad \qquad \qquad \qquad +(-1)^{\epsilon_{k+1}}p  [a_{1},\cdots, a_{k}]dm \\
& & +(-1)^{|p|} pa_{1}[a_{2},\cdots , a_{k}]m +\sum_{i=2}^{k} (-1)^{\epsilon_{i}}p  [a_{1},\cdots,a_{i-1}a_{i},\dots a_{k}]m \\
& & \qquad \qquad \qquad \qquad \qquad \qquad \qquad \qquad - (-1)^{\epsilon_{k}}p[a_{1},\cdots ,a_{k-1} ]a_{k}m,
\end{eqnarray*}
where
$$
\epsilon_{i}=|p|+ |a_{1}|  +\cdots |a_{i-1}|-i+1.
$$
\end{itemize}

Let $P=A$ and  $\epsilon_M:B(A,A,M)\rightarrow M$ be defined by

\begin{equation}
\epsilon_M(a [a_{1}, \cdots , a_{k}])m)=\begin{cases}0, \text{ if }  k\geq 1,\\ am, \text{ if } k=0.
\end{cases}
\end{equation}
It is clear that $\epsilon_M$ is a map of left $A$-modules if $M$.
\begin{lemma} In the category of left $A$-modules,  $\epsilon_M:B(A,A,M)\rightarrow M$ is a semi-free resolution.
\end{lemma}
\begin{proof}
We first prove that this is a resolution. Let $h:B(A,A,M)\rightarrow B(A,A,M)$ be defined by
\begin{equation}
h(a[a_1,a_2,\cdots a_k]m)= \begin{cases}[a,a_1,\cdots a_k]m, \text{ if }  k\geq 1,\\
[a]m,\text{ if } k=0.
\end{cases}
\end{equation}
On can easily check that $ [D,h]= id$  on $\ker \epsilon_M$, which implies $H_*(\ker (\epsilon_M))=0$. Since $\epsilon_M$ is surjective,  $\epsilon_M$ is a quasi-isomorphism.
Now we prove that $B(A,A,M)$ is a semifree $A$-module.  Let
$F_k= \bigoplus_{i\leq k} A\otimes T(s\bar{A})^{\otimes i} \otimes M$.  Since $d_1(F_{k+1})\subset F_k$, then $F_{k+1}/F_k$ is isomorphic as a differential graded $A$-module to $ (A\otimes (sA)^{\otimes k}\otimes M, d_0)=(A,d)\otimes_\kk ( (sA)^{\otimes k},d)\otimes (M,d)$. The latter is a semifree $(A,d)$-module since $( (sA)^{\otimes k},d)\otimes_\kk(M,d)$ is a semifree $\kk$-module via the filtration
$$
0 \hookrightarrow \ker (d\otimes 1+1\otimes d) \hookrightarrow ( (sA)^{\otimes k},d)\otimes_\kk(M,d).
$$
Therefore $B(A,A,M)$ is semi-free by Lemma~\ref{lemm-filt}.

\end{proof}

\begin{corollary}\label{cor-reso}  The map $\epsilon_A: B(A,A):=B(A,A,\kk)\rightarrow \kk$ given by
\begin{equation*}
\epsilon_k(a[a_1,a_2\cdots a_n])=\begin{cases} \epsilon(a), \text{ if } n=0,\\ 0 \text{ otherwise}
 \end{cases}
 \end{equation*}
 is a resolution. Here $\epsilon:A\rightarrow \kk$ is the augmentation of $A$. In other words $ B(A,A)$ is acyclic. 
\end{corollary}
\begin{proof}   In the previous lemma, let $M=\kk$ be the differential $A$-module with trivial differential and the module structure
$a.k:= \epsilon(a)k$.
\end{proof}

\begin{lemma} In the category $Mod(A^e)$,  $\epsilon_A:B(A,A,A)\rightarrow A$ is a semifree resolution.
\end{lemma}
\begin{proof}
The proof is similar to  the proof of the previous lemma. First of all, it is obvious that this is a map of  $A^e$-modules. Let
$F_k= \bigoplus_{i\leq k} A\otimes T(s\bar{A})^{\otimes i} \otimes A$. Then  $F_{k+1}/F_k$ is isomorphic as a differential graded $A$-module to $ (A\otimes (sA)^{\otimes k}\otimes A, d_0)=(A,d)\otimes_\kk ( (sA)^{\otimes k},d)\otimes (A,d)$.  The latter is semi-free as $A^e$-module since $( (sA)^{\otimes k},d)$ is a semi-free $\kk$-module via the filtration $\ker d \hookrightarrow   (sA)^{\otimes k}$.
\end{proof}

Since  the two-sided bar construction $B(A,A,A)$ provides us with a semi-free resolution of $A$  we have that
$$
HH_*(A,M)=H_* (B(A,A,A)\otimes_{A^e}M)= \Tor_*^{A^{e}}(A,M)
$$
and
$$
HH^*(A,M)=H^* (Hom_{A^e}(B(A,A,A), M))= \Ext^*_{A^{e}}(A,M).
$$

In some special situations, for instance that of  Calabi-Yau algebras, one can choose smaller resolutions to compute Hochschild homology or cohomology.

The following result will be useful.

\begin{lemma}\label{lem-util} If $H^*(A)$ is finite dimensional then for all finitely generated semi-free $A$-bimodules $P$ and $Q$, $H^*(P)$, $H^*(Q)$ and $H^*(\Hom_{A^e}(P,Q))$ are also finite dimensional.

\end{lemma}
\begin{proof}
Since $A$ has finite dimensional cohomology, we see that $H^*(A\otimes A^{op})$ is finite dimensional. Similarly $P$ (or $Q$) has finite cohomological dimension since it is obtained via a finite sequence of extensions of free bimodules of the form $(A\otimes A^{op})[n]$. We also have $\Hom_{A^{e}}(A\otimes A^{op},A\otimes A^{op})\simeq A\otimes A^{op}$, and $A\otimes A^{op}$ is a free $A$-bimodule of finite cohomological dimension. Since $\Hom_{A^e}(P,Q)$ is obtained through a finite sequence of extensions of shifted free $A$-bimodules, we obtain that it has  finite cohomological dimension.
\end{proof}

\section{Calabi-Yau DG algebras} \label{sec:CY}

Throughout this section $(A,d)$ is a differential graded algebra, and by an $A$-bimodule we mean a differential graded $(A,d)$-bimodule.

In this section we essentially explain how an isomorphism
\begin{equation}\label{vdb}HH^*(A,A)\simeq HH_*(A,A)\end{equation}
(of $HH^*(A,A)$-modules) gives rise to a BV-structure on $HH^*(A,A)$ extending its canonical Gerstenhaber structure. For a Calabi-Yau DG algebra one does have such an isomorphism (\ref{vdb}) and this is a special case of a more general statement due to Van den Bergh~\cite{VdBergh}. The main idea is due to Ginzburg~\cite{Ginz} who proved that, for a Calabi-Yau algebra $A$,  $HH^*(A,A)$ is a BV-algebra. However he works with ordinary algebras rather than DG algebras. Here we have adapted his result to the case of Calabi-Yau DG algebras. For this purpose one has to work in the correct derived category of $A$-bimodules, and this is the derived category of perfect $A$-bimodules as formulated below. All this can be extended to the case of $A_\infty$ but for simplicity we refrain from doing so.

\subsection{Calabi-Yau algebras}

The notion of a Calabi-Yau algebra was first introduced by Ginzburg~\cite{Ginz} in the context of graded algebras, and then generalized by Kontsevich-Soibelman~\cite{KS}  to the context of differential graded algebras.

\begin{definition} (Kontsevich-Soibelman~\cite{KS})
\begin{numlist}
\item An $A$-bimodule is \emph{perfect} if it is quasi-isomorphic to a direct summand of a finitely generated  semifree $A$-bimodules.
\item $A$ is said to be \emph{homologically smooth} if it is perfect as an $A$-bimodule.
\end{numlist}
\end{definition}

\begin{remark} In~\cite{KS}, the definition of perfectness uses the notion of \emph{extension}~\cite{Keller} and it is essentially the same as ours.
\end{remark}
We define \emph{DG-projective $A$-modules} to be the direct summands of semifree $A$-modules. As a consequence, an $A$-bimodule is perfect if and only if it is quasi-isomorphic to a finitely generated DG-projective $A$-bimodule. We call the latter a finitely generated \emph{DG-projective $A$-module resolution}. This is analogous to having a bounded projective resolution in the case of ordinary modules (without differential). By Proposition~\ref{prop-ess}, DG-projectives have all the nice homotopy theoretic properties that one expects.

The content of the next lemma is that $A^!:=\RHom_{A^{e}}(A,A^e)$, called the \emph{derived dual} of $A$, is also a perfect $A$-bimodule. The $A$-bimodule structure of $A^!$ is induced by the right action of $A^e$ on itself. Recall that for an $A$-bimodule $M$, $\RHom_{A^{e}}(-,M)$ is the right derived functor of $\Hom_{A^{e}}(-,M)$, i.e.  for an $A$-bimodule $N$, $\RHom_{A^{e}}(N,M)$ is the complex $\Hom _{A^{e}}(P,M)$ where $P$ is a DG-projective $A$-bimodule quasi-isomorphic to $N$. In general, $M^!=\RHom_{A^{e}}^*(M,A^{e})$ is different from the usual  dual $$M^\vee=\Hom_{A^{e}} (M,A^{e}).$$ 
\begin{lemma}\label{dual-perfect}\cite{KS}
If A is homologically smooth then $A^!$ is a perfect $A$-bimodule.
\end{lemma}
\begin{proof} Let $P=P_i  \twoheadrightarrow A$ be a finitely generated DG-projective resolution. Note that $A^!= \RHom_{A^e}(A, A^e)$ is quasi-isomorphic to the complex $\Hom_{A^e}(P_i, A^e)$. Each $P_i$ being a direct summand of a semi-free module $Q_i$.  Since $Q_i$ is obtained through  a finite sequence of extensions of a free $A^e$-modules, 
$\Hom(Q_i, A^e)$  is also a semi-free module. Clearly $\Hom_{A^e}(P_i, A^e)$ is a direct summand of the semi-free module $\Hom(Q_i, A^e)$, therefore DG-projective.
This proves the lemma.

\end{proof}

We say that a DG algebra $A$ is \emph{compact} if the cohomology $H^*(A)$ is finite dimensional.

\begin{lemma} \label{bad-lemma} A compact homologically smooth DG algebra $A$ has finite dimensional Hochschild cohomology $HH^*(A,A)$.
\end{lemma}
\begin{proof}By assumption $A$ has finite dimensional cohomology and so does $A^e=A\otimes A^{op}$. Now let $P\twoheadrightarrow A$ be a finitely generated DG-projective resolution of $A$-bimodules. We have a quasi-isomorphism of complexes 
$$
C^*(A,A)\simeq \RHom_{A^{e}} (A,A)\simeq Hom_{A^{e}} (P,P),
$$ 
and the last term has finite dimensional cohomology by Lemma~\ref{lem-util}.
\end{proof}

\begin{definition}(Ginzburg~\cite{Ginz}, Kontsevich-Soibelman~\cite{KS})
A $d$-dimensional \emph{Calabi-Yau differential graded algebra} is a homologically smooth DG-algebra  endowed with an $A$-bimodule quasi-isomorphism
\begin{equation}\label{cond1}
\psi:A\overset{\simeq}{\longrightarrow} A^{!}[d]
\end{equation}
such that 
\begin{equation}\label{cond2}
\psi ^!=\psi [d].
\end{equation}
\end{definition}

The main reason to call such algebras ``Calabi-Yau'' is that the algebra $\mathrm{End}(\mathcal{E})$ of endomorphisms of a tilting generator $\mathcal{E}\in D^b(Coh (X))$ of the bounded derived category of coherent sheaves on a smooth algebraic variety $X$ is a Calabi-Yau algebra if and only if  $X$  is a Calabi-Yau (see~\cite{Ginz} Proposition 3.3.1).  

There are many other examples provided by representation theory. For instance most of the three dimensional Calabi-Yau algebras are obtained as quotients of the free associative algebras $F=\mathbb{C}\langle x_1,\cdots x_n\rangle$ on $n$ generators. An element $\Phi$ of $F_{cyc}:=F/[F,F]$ is called a \emph{cyclic potential}.  One can define the partial derivatives $\frac{\partial}{\partial x_i}: F_{cyc} \rightarrow F$ in this setting. Many of the 3-dimensional Calabi-Yau  algebras are obtained as a quotient $ \mathcal{U}(F, \phi)= F/\{ \frac{\partial \Phi}{\partial x_i=0}_{i=0,\cdots n}\} $. For instance, for $\Phi(x,y,z)= xyz-yzx$ we obtain $\mathcal{U}(F, \phi)=\mathbb{C}[x,y,z]$, the polynomial algebra in $3$ variables. The details of this discussion are not relevant to the context of this chapter, which is that of algebraic models of free loop spaces. We therefore refer the reader to~\cite{Ginz} for further details.

\medskip

In the above definition $A^!=\RHom_{A^{e}}(A,A^{e})$ is called the \emph{dualizing bimodule}. This is an $A$-bimodule using outer multiplication. Condition (\ref{cond1}) amounts to the following.
\begin{proposition}\label{ext-tor}(Van den Bergh Isomorphism~\cite{VdBergh}) Let $A$ be a Calabi-Yau DG algebra  of dimension $d$.  Then for all $A$-bimodules $N$ we have
\begin{equation} \label{iso-VdB}
HH_{d-*}(A, N) \simeq  HH^*(A,N).
\end{equation}
\end{proposition}
\begin{proof} We compute 
\begin{equation}
\begin{split}
HH^*(A,N)&\simeq \Ext^*_{A^e}(A,N)\simeq H^*(\RHom_{A^e}(A,N)) \\
& \simeq H^*(\RHom_{A^e}(A,A^e)\otimes^L_{A^e} N) \\ & \simeq H^* (A^!\otimes^L_{A^e}N) \simeq H_*(A[-d]\otimes_{A^e}^L  N) \\
& \simeq \Tor_*^{A^e}(A,N)\simeq HH_{d-*}(A,N) \\
\end{split}
\end{equation}
\end{proof}

Note that the choice of the $A$-bimodule isomorphism $\psi$ is important and it is characterized by the image of the unit  $\pi=\psi (1_A)\in A^!$. By definition, $\pi$ is the \emph{volume of the Calabi-Yau algebra $A$}. For a Calabi-Yau algebra $A$  with a volume $\pi$ and $N=A$, we obtain an isomorphism

\begin{equation}
D=D_\pi: HH_{d-*}(A, A) \rightarrow  HH^*(A,A).
\end{equation}

\noindent One can use $D$ to transfer the Connes operator $B$ from $ HH_{*}(A, A)$ to $HH^*(A,A)$, via
$$
\Delta=\Delta_{\pi}:= D\circ B \circ D^{-1}
$$

In the following lemmas $A$ is a Calabi-Yau algebra with a fixed volume $\pi$ and the associated isomorphism $D$.

\begin{lemma}\label{lem0} $f \in HH^*(A,A)$ and $a\in HH_*(A,A)$
$$
D(i_fa)=f\cup D a.
$$

\end{lemma}
\begin{proof}
To prove the lemma we use the derived description of Hochschild homology and cohomology, and of the cap and cup product. Let $(P,d)$ be a projective resolution of $A$. Then $HH_*(A,A)$ is computed by the complex $(P\otimes_{A^{e}} P, d)$, and similarly $HH^*(A,A)$ is computed by the complex $(\End (P), \ad(d)= [-,d])$. Here $\End (P)=\oplus_{r\in \Z} \Hom_{A^{e}}(P, P[r])$. Then the cap product corresponds to the natural pairing
\begin{equation}
ev: (P\otimes_{ A^e} P ) \otimes_\kk \End (P) \rightarrow P\otimes_{A^e} P
\end{equation}
given by $ev: (p_1\otimes p_2)\otimes f\mapsto p_1\otimes f(p_2)$. The quasi-isomorphism $\psi: A\rightarrow A^![d]$ yields a morphism $\phi: P\rightarrow P^\vee$.  Let us explain this in detail.

Using the natural identification $P^\vee \otimes _{A^e}P = \End(P)$, we have a commutative diagram

\begin{equation}\label{diagcom1}
\xymatrix{  \End (P) \otimes_\kk  (P\otimes_{ A^e} P ) \ar[rr]^-{ev}\ar[d]^{id \otimes(\phi\otimes id)}&&  P\otimes_{A^e} P \ar[d]^-{\phi\otimes\id}\\
\End (P) \otimes_\kk \End (P)\ar[rr]^-{\text{composition}} & & \End (P)=P^\vee \otimes _{A^e}P}
\end{equation}
The evaluation map  is defined by $ev (\psi\boxtimes (x\otimes y)):= x\otimes \psi (y)$. After passing to (co)homology, $\phi\otimes\id$ becomes $D$, the composition induces the cup product and $ev$ is the contraction (cap product), hence
$$
D(i_f a)= f\cup Da
$$
which proves the lemma.
\end{proof}

\begin{lemma}\label{lem1}  For $f, g\in HH^*(A,A)$ and $a\in HH_*(A,A)$ we have
\begin{equation*}
\begin{split}
[f,g]\cdot a&= (-1)^{|f|} B((f\cup g)a)-f\cdot  B(g\cdot a)+ (-1)^{(|f|+1)(|g|+1)}g\cdot B(f\cdot a)\\& +(-1)^{|g|} (f\cup g)\cdot B(a).
\end{split}
\end{equation*}
\end{lemma}

\begin{proof} We compute     
\begin{equation}\label{BV1}
    \begin{split}
       i_{[f,g]}&=L_f i_g - i_gL_f   \\
        &= i_fBi_g-Bi_fi_g -i_gBi_f-i_gi_fB\\
        &=i_fBi_g-i_gBi_f-Bi_{f\cup g}-i_gi_fB
    \end{split}
\end{equation}
\end{proof}

\begin{corollary}\label{cor1}
For all $f,g\in  HH^*(A,A)$ and $a\in HH_*(A,A)$, we have
\begin{equation*}
\begin{split}
[f,g]\cup D(a)&=(-1)^{|f|} \Delta(f\cup g\cup D a) -f\cup \Delta (g\cup D(a))\\& +(-1)^{(|f|+1)(|g|+1)}g\cup \Delta (f\cup D(a))- f \cup g \cup  DB(a).
\end{split}
\end{equation*}

\end{corollary}
\begin{proof}

After applying $D$ to the identity of the previous lemma we obtain
\begin{equation*}
\begin{split}
D(i_{[f,g]}a)& =(-1)^{|f|} D(Bi_{f\cup g}a)-D(i_fBi_g(a))+(-1)^{(|f|+1)(|g|+1)}D(i_gBi_f(a)) \\ &+(-1)^{|g|}D(i_fi_gB(a)).
\end{split}
\end{equation*}
By Lemma~\ref{lem0} and $DB= \Delta D $, this reads
\begin{equation*}
\begin{split}
[f,g]\cup D(a)= & \, (-1)^{|f|} \Delta D(i_{f\cup g}a) -f\cup \Delta D (i_g(a))\\
&  +(-1)^{(|f|+1)(|g|+1)}g\cup \Delta D(i_f(a)) - g\cup f \cup  DB(a).
\end{split}
\end{equation*}
Once again by Lemma~\ref{lem0} we get
\begin{equation*}
\begin{split}
[f,g]\cup D(a)=&\, (-1)^{|f|} \Delta(f\cup g\cup D a)  -f\cup \Delta (g\cup D(a))\\&+(-1)^{(|f|+1)(|g|+1)}g\cup \Delta (f\cup D(a))+(-1)^{|g|} f\cup g \cup  DB(a).
\end{split}
\end{equation*}

\end{proof}

\begin{theorem} \label{main-theo1} Let $A$ be a Calabi-Yau algebra with volume $\pi\in A^!$. Then $(HH^*(A,A), \cup, \Delta)$ is a BV-algebra, i.e. 
\begin{equation} \label{BV-cond-bis}
[f,g]=(-1)^{|f|}\Delta(f\cup g)-(-1)^{|f|}\Delta(f)\cup g-f\cup \Delta(g).
\end{equation}

\end{theorem}

\begin{proof}
In the statement of the previous lemma choose  $a\in HH_d(A,A)$ such that $D(a)=1\in HH^0(A,A)$. The identity (\ref{BV-cond-bis})  follows since $B(a)=0$ for  obvious degree reasons.
\end{proof}

\subsection{Poincar\'e duality groups and chains on the Moore based loop space}\label{secmoore}

Let us finish this section with some interesting examples of  DG Calabi-Yau algebras.  We will also discuss chains on the Moore based loop space. This example plays an important role in symplectic geometry in relation with a generation result for a particular type of Fukaya category  called \emph{the wrapped Fukaya category} (see~\cite{abou} for more details). One can then compute the Hochschild homology of wrapped Fukaya categories using the Burghelea-Fiedorowicz-Goodwillie theorem. This theorem implies that the Hochschild homology of the Fukaya category of a closed oriented manifold is isomorphic to the homology of the free loop space of the manifold.   
 
We start with a more elementary example, namely that of \emph{Poincar\'e duality groups}, which include fundamental groups of closed oriented aspherical manifolds. Closed oriented irreducible 3-manifolds are aspherical, hence they provide us with a large and interesting class of Poincar\'e duality groups.

\begin{proposition} \label{pro-PDG}Let $G$ be finitely generated oriented Poincar\'e duality group of dimension $d$. Then
$\kk[G]$ is a Calabi-Yau algebra of dimension $d$, therefore $HH^*(\kk[G],\kk[G])$ is a BV-algebra.
\end{proposition}
\begin{proof}

First note that $\kk[G]$ is only an ordinary algebra without grading and differential. The hypothesis that
$G$ is a finitely generated oriented Poincar\'e duality group of dimension $d$ means that $\kk$ has a bounded finite projective resolution $P=\{P_d\rightarrow \cdots P_1\rightarrow P_0\} \twoheadrightarrow \kk$ as a left $\kk[G]$-module, and $H^d(G,\kk[G])\simeq \kk$ as a $\kk[G]$-module. Here $\kk$ is equipped with the trivial action and $\kk[G]$ acts on $H^d(G,\kk[G])$ from the left via the coefficient module. In particular we have
\begin{equation}
 \Ext_{\kk[G]}^i (\kk,\kk[G]) \simeq  \left\{
      \begin{array}{ll}
          \kk , & i=d,\\
         0 , & \hbox{otherwise.}
      \end{array}
    \right.
\end{equation}
In other words the resolution $P$ has the property that $P^\vee:= \Hom_{\kk[G]} (P,\kk[G])$, after a shift in degree by $d$, is also a resolution of $\kk$ as a right $\kk[G]$-module (see~\cite{Brown} for more details).

Note that using the map $g\to (g,g^{-1})$ we can turn $\kk[G]^e$ into a right $\kk[G]$-module.  More precisely $(g_1\otimes g_2)g:=g_1g\otimes g^{-1}g_2$. The tensor product $\kk[G]^e\otimes_{\kk[G]} \kk$ is isomorphic to  $\kk[G]$  as a left $\kk[G]^e$-module. The isomorphism is given by $(g_1\otimes g_2)\boxtimes 1\mapsto g_1g_2$. Similarly $\kk[G]^e$ can be considered as a left $\kk[G]$-module using the action $g(g_1\otimes g_2):=gg_1\otimes g_2g^{-1}$ and once again 
 $\kk\otimes_{\kk[G]} \kk[G]^e\simeq \kk[G]$ as a  right $\kk[G]^e$-module using the isomorphism  $1 \boxtimes (g_1\otimes g_2)\mapsto g_2g_1$.
 
It is now clear that  $\kk[G]^e \otimes_{\kk[G]}P_d\rightarrow \cdots \kk[G]^e \otimes_{\kk[G]}P_1 \rightarrow \kk[G]^e \otimes_{\kk[G]}P_0 \twoheadrightarrow \kk[G]\otimes_{\kk[G]} \kk \simeq \kk[G]$ is a projective resolution of $\kk[G]$ as a $\kk[G]$-bimodule, proving that $\kk[G]$ is homologically smooth.  Similarly $\Hom_{\kk[G]} (P,A)\otimes _{\kk[G]}\kk[G]^e$ is a projective resolution of $\kk[G]$ as $\kk[G]$-bimodule. Therefore we have a homotopy equivalence
$$
\phi: \kk[G]^e\otimes_{\kk[G]}P\to\Hom_{\kk[G]} (P,A)\otimes _{\kk[G]}\kk[G]^e  .  
$$

Now if we take $Q= \kk[G]^e\otimes_{\kk[G]}P$ as a DG-projective resolution of  $\kk[G]$ as $\kk[G]$-bimodule, then
$ \kk[G]^!= \Hom_{\kk[G]^e}( \kk[G]^e\otimes_{\kk[G]}P, \kk[G]^e ) \simeq \Hom_{\kk[G]}( P,\kk[G]^e)$, where the right $\kk[G]$-module structure of $\kk[G]^e$ was described above. Since $P$ is $\kk[G]$-projective, the natural map  
$$ 
\Hom_{\kk[G]}( P,\kk[G])\otimes _{\kk[G]} \kk[G]^e \overset{\simeq}{\longrightarrow }\Hom_{\kk[G]}( P,\kk[G]^e)
$$ 
is an isomorphism. Therefore $\phi$ is nothing but an equivalence $\kk[G]\overset{\phi}{\simeq} \kk[G]^!$ in the derived category of $\kk[G]$-bimodules. It remains to prove that $\phi^!\simeq \phi[d]$ in the derived category of $\kk[G]$-bimodules.  We have
\begin{equation}
\begin{split}
\phi^!=\phi^\vee: \Hom_{\kk[G]^e}( \Hom_{\kk[G]}&(P,\kk[G])\otimes_{\kk[G]}  \kk[G]^e,   \kk[G]^e) \\
& \rightarrow  \Hom_{\kk[G]^e}(\kk[G]^e\otimes_{\kk[G]}P,\kk[G]^e)\simeq \kk[G]^!.
\end{split}
\end{equation}

On the other hand we have the natural inclusion map
\begin{equation}
\begin{split}
i: \kk[G]^e\otimes_{\kk[G]}P\to &\Hom_{\kk[G]^e}(\Hom_{\kk[G]^e}(\kk[G]^e\otimes_{\kk[G]}P, \kk[G]^e))\\ &\simeq  \Hom_{\kk[G]^e}( \Hom_{\kk[G]}(P,\kk[G])\otimes_{\kk[G]}  \kk[G]^e,   \kk[G]^e)
\end{split}
\end{equation}
and one can easily check that $\phi^!\circ i=\phi$ after a shift in degree by $d$. This proves that $\phi^! \simeq \phi[d] $ in the derived category.

\end{proof}

\begin{remark}
In the case when $G= \pi_1(M)$ is the fundamental group of an aspherical manifold $M$, Vaintrob~\cite{Vaint} has proved that the BV-structure on the Hochschild (co)homology $HH^{*+d}(\kk[G],\kk[G]) \simeq  HH_*(\kk[G],\kk[G])$ corresponds to the Chas-Sullivan BV-structure on $H_*(LM,\kk) $.
\end{remark}

Let $(X,\ast)$ be a pointed finite CW complex with Poincar\'e duality. The Moore loop space of $X$,  $\Omega X=\{\gamma:[0,s]\to X \, | \, \gamma(0)=\gamma(s)=*, \, s\in \R^{>0}\}$ is equipped with the standard concatenation, which is strictly associative. Therefore the cubic chains $C_*(\Omega X)$ can be made into a strictly associative algebra using the Eilenberg-Zilber map and the concatenation.
In~\cite{Ginz} there is a  sketch of proof that $C_*(\Omega X)$ is homologically smooth. Here we prove more using a totally different method.  

\begin{proposition} For a Poincar\'e duality finite CW-complex $X$, $C_*(\Omega X)$ is a Calabi-Yau DG algebra.
\end{proposition}
\begin{proof} The main idea of the proof  is essentially taken from~\cite{FT4}. Let $A=C_*(\Omega X)$ be the complex of cubic singular chains on  the Moore loop space. By composing the Eilenberg-Zilber and concatenation maps $C_*( \Omega X) \otimes C_*(\Omega X) \overset{EZ}{\longrightarrow}  C_*(\Omega X\times \Omega X)
\overset{concaten.}{\longrightarrow} C_*(\Omega X)$ one can define an associative product on $A$ . This product is often called \emph{the Pontryagin product}.
One could switch to the simplicial singular chain complex of the standard based loop space $\{\gamma:[0,1]\to X\, | \, \gamma(0)=\gamma(1)=*\}$,
but then one would have to work with $A_\infty$-algebras, $A_\infty$-bimodules and their derived category. All these work nicely~\cite{KS,Malm} and the reader may wish to write down the details in this setting. 

Note that $A$ has additional structure. First of all, the composition of the Alexander-Whitney and of the diagonal maps $C_*(\Omega X)\overset{diagonal}{\longrightarrow}  C_*(\Omega X\times \Omega X)
\overset{A-W}{\longrightarrow} C_*(\Omega X)\otimes C_*(\Omega X)$ endows $A$ with a coassociative coproduct, which together with the Pontryagin product makes $C_*(\Omega X)$ into a bialgebra.
One can consider the inverse map on $\Omega X$ which makes the bialgebra $C_*(\Omega X)$ into a differential graded Hopf algebra up to homotopy. In order to get a strict differential graded Hopf algebra, one constructs a topological group $G$ which is homotopy equivalent to $\Omega X$ (see~\cite{Kan,HessTonks}).  This can be done, and one can even find a simplicial topological group homotopy equivalent to $\Omega X$. Therefore from now on we assume that $\Omega  X=G $ is a topological group and $C_*(\Omega X) $ is a differential graded Hopf algebra $(A, \cdot, \delta, S)$ with coproduct $\delta$ and antipode map $S$.  

First we prove that $A$ has a finitely generated semifree resolution as  an $A$-bimodule.  The proof which is essentially taken from~\cite{FT4}  (Proposition 5.3) relies on the cellular structure of $X$.
Consider the path space $E=\{\gamma:[0,s]\rightarrow X \, |Ê\, \gamma(s)=\ast\}$.  Using the concatenation of paths and loops one can define an action of $\Omega X$ on $E$, and thus $C_*(E)$  becomes a $C_*(\Omega X)$-module. This action translates to an action of $A$ on $E$ which is from now on an $A$-module. 
Let $ G=\Omega X \rightarrow E\rightarrow X$ be the path space fibration of $X$.  We will construct a finitely generated semifree resolution  of $C_*(E)$ as an $A$-module  which,  since $E$ is contractible, provides us with a finitely generated semifree resolution of  $\kk\simeq C_*(E)$.  Now  by tensoring this resolution with $A^e$ over $A$ we obtain a finitely generated semifree resolution of $A$ as $A$-bimodule.  Here the $A$-module structure of $A^e$  is defined via the composite $Ad_0:= (A\otimes S) \delta : A\to A^e$, similar to the case of Poincar\'e duality groups (see the proof Proposition~\ref{pro-PDG}).

The semifree resolution of $ C_*(E)$ is constructed as follows. Let $X_1\subset X_2 \subset \cdots \subset X_m$ be the skeleta of $X$,  let $D_n=\coprod D_\alpha^n$ be the disjoint union of its $n$-cells, and let $\Sigma_n=\coprod S_\alpha^{n-1}$ be the disjoint union of their boundaries.
Let $V_n=H_*(X_n,X_{n-1})$ be the free $\kk$-module on the basis $v_\alpha^n$. Using the cellular structure of $X$  we construct an $A=C_*(G)$-linear quasi-isomorphism $\phi: (V\otimes A)\rightarrow C_*(E)$ where
$V=\oplus_n V_n$, inductively from the restrictions $\phi_n=\phi | \oplus_{i\leq n} V_i \otimes A\rightarrow C_*(E_n)$. The induction step $n-1$ to $n$  goes as follows. Let $f: (D_n,\Sigma_n )\rightarrow  (X_n,X_{n-1})$ be the characteristic map. Since the (homotopy) $G$-fibration $E$ can be trivialized over $D^n$, one has a homotopy equivalence of pairs
$$
\Phi: (D_n,\Sigma_n)\times G\rightarrow (E_n,E_{n-1}),
$$
where $E_i= \pi^{-1}(X_i)$. We have a commutative diagram 
\begin{equation}
\xymatrix{  &   C_*(E_n)\ar[d]_q \\ C_*(D_n,\Sigma_n)\otimes C_*(G) \ar[d]\ar[r]^-{ \Phi_*\circ EZ} & C_*(E_n,E_{n-1})\ar[d]_-{\pi_*} \\
                 C_*(D_n,\Sigma_n) \ar[r]_-{f_*} & C_*(X_n,X_{n-1}) }
\end{equation}
whose horizontal arrows are quasi-isomorphisms. Here $q$ is the standard projection map and $\pi$ is the fibration map. Since $q$ is surjective there is an element $w^n_\alpha \in C_*(E_n)$ such that
$q_*(w_n^\alpha)= \Phi_*\circ EZ(v^\alpha_n\otimes 1)$. Since $ v^\alpha_n\otimes 1$ is a cycle  we have that $dw_n^\alpha \in C_*(E_{n-1})$. Because we have assumed that
 $m_{n-1}$ is a quasi-isomorphism, there is a cycle $z^\alpha_{n-1}\in \oplus_{i\leq n-1} V_i \otimes A$ such that $\phi_{n-1}(z^\alpha_{n-1})=dw_ n^\alpha$. First we extend the differential by 
 $d(v_\alpha\otimes 1)=z_\alpha$. We extend $\phi_{n-1}$ to $\phi_n$ by defining
 $\phi_n(v^\alpha_n\otimes 1 )=  w_\alpha$.  The fact that $\phi_n $ is an quasi-isomorphism follows from an inductive argument combined with the Five Lemma, using the fact that on the quotient $\phi_n: V_n\otimes  C_*(G)\rightarrow  C_*(E_n, E_{n-1})$ is a quasi-isomorphism.

Next we prove that $A\simeq A^!$ in the derived category of $A^e$-bimodules, which is a translation of Poincar\'e duality.  Let $Ad_0:A\rightarrow A^e$ be defined by $Ad_0= (id\otimes S) \delta$. For an $A$-bimodule $M$ let $Ad_0^*(M)$ be the $A$-module
whose $A$-module structure is induced using pull-back by $Ad_0$. By applying the result of F\'elix-Halperin-Thomas on describing the chains of the base space of a $G$-fibration $G \rightarrow EG\rightarrow BG\simeq X$,  we get a quasi-isomorphism of coalgebras
$$
C_*(X)\simeq B(\kk,A,\kk).
$$
Note that  $ B(\kk,A,\kk)\simeq B(\kk, A,A)\otimes_A B(A,A, \kk)$.  Poincar\'e duality for  $X$ implies that there is a cycle  $z_1\in C_*(X)$ such that  capping with $z_1$
\begin{equation}\label{poin}
-\cap z_1: C^*(X)\rightarrow C_{*-d}(X)
\end{equation}
is a quasi-isomorphism. The class $z_1$ corresponds to a cycle $z\in  B(\kk, A,A)\otimes_A B(A,A, \kk)$ and the quasi-isomorphism (\ref{poin}) corresponds to the quasi-isomorphism
\begin{equation}
ev_{z,P}: \Hom_k ( B(\kk, A,A) , P) \rightarrow B(A,A, \kk)\otimes P
\end{equation}
given by $ ev_{z}(f)= \sum f(z_i)z'_i $, where $f\in \Hom_k ( B(\kk, A,A) , \kk) $ and $z=\sum z_i\otimes z'_i$.

Let $E=Ad_0^*(A^e)$.  We have the following quasi-isomorphisms of  $A^e$-modules:
\begin{equation}
\begin{split}
A&\simeq B(A,A,A)\simeq B(Ad^*(A^e),A,\kk)\simeq E \otimes_A B(A,A,\kk),
\end{split}
\end{equation}
where $A^e=A\otimes A $ acts on the last term from the left and on the factor $E$. On the other hand
\begin{equation}
\begin{split}
A^!&\simeq \Hom_{A^e}(B(A,A,A),A^e)\simeq \Hom_{A^e} (B(\kk,A,Ad^*(A^e)), A^e)\\ &\simeq \Hom_{A} (B(\kk,A,A),\Hom_{A^e}(Ad_0^*(A^e), A^e) )\\
& \simeq \Hom_{A} (B(\kk,A,A),Ad_0^*(A^e)).
\end{split}
\end{equation}
Therefore $ev_{z,E}$ is a quasi-isomorphism of $A^e$-modules  from $A^!$ to $A[-d]$. 
\end{proof}

\begin{corollary}  For a closed oriented manifold $M$,  $HH^*(C^*(\Omega M), C^*(\Omega M))$  is a BV-algebra.
\end{corollary}
\begin{proof}
Note that in the proof of Theorem~\ref{main-theo1} we don't use the second part of the Calabi-Yau condition. We only use the derived equivalence $A\simeq A^![d]$.
\end{proof}

\begin{remark}
Recently Malm~\cite{Malm} has proved that the Burghelea-Fiedorowicz-Goodwillie (BFG) isomorphism  (\cite{BurFied,Good})
$$HH^*( C_*(\Omega M),C_*(\Omega M) )\simeq HH_*(C_*(\Omega M),C_*(\Omega M))\overset{\text{{\tiny BFG}}}{\simeq} H_*(LM)$$ 
is an isomorphism of  BV-algebras, where $H_*(LM)$ is equipped with  the Chas-Sullivan BV-structure~\cite{CS1}. 
\end{remark}

\section{Derived Poincar\'e duality algebras} \label{sec:Poincare}
In this section we show how an isomorphism
\begin{equation*}HH^*(A,A)\simeq HH^*(A,A^\vee)\end{equation*}  of $HH^*(A,A)$-modules gives rise to a BV-structure on $HH^*(A,A)$ whose underlying Gerstenhaber structure is the canonical one. The next lemma follows essentially from Lemma~\ref{lem1}.

\begin{lemma}\label{lem2}  For $a, b\in HH^*(A,A)$ and $\phi\in HH^*(A,A^{\vee})$ we have
\begin{equation}\label{lem2-eq}
\begin{split}
[f,g]\cdot \phi= & \, (-1)^{|f|} B^{\vee}((f\cup g)\phi)-f \cdot B^{\vee}(g \cdot \phi)\\
& + (-1)^{(|f|+1)(|g|+1)}g\cdot B^\vee(f \cdot \phi) +(-1)^{|g|} (f\cup g) \cdot B^{\vee}(\phi).
\end{split}
\end{equation}
\end{lemma}
\begin{proof}
To prove the identity, one evaluates the cochains in 
$$
C^* (A,A^\vee)=\Hom_{\kk}(T(s\bar{A}),A^\vee)
\simeq  \Hom_{\kk}(A\otimes T(s\bar{A}),\kk)
$$ 
on both sides on a chain  $ x=a_0[a_1,\cdots, a_n] \in A\otimes T(s\bar{A})$.  By (\ref{contr-2}) and Lemma~\ref{lem1} we have:
\begin{equation*}
 \begin{split}
 ([f,g]\cdot \phi)(x)& =(i_ {[f,g]}\phi)(x)=(-1)^{|[f,g]|\cdot |\phi|}\phi(i_{[f,g]}(x))\\
&=(-1)^{|[f,g]|\cdot |\phi|} \phi((-1)^{|f|}B((f\cup g)\cdot x)-f\cdot  B(g \cdot x) \\
& \, \, \, + (-1)^{(|f|+1)(|g|+1)}g \cdot B(f \cdot x)+(-1)^{|g|} (f\cup g) \cdot B(x))\\
&=(-1)^{|[f,g]| \cdot |\phi|+|f|+ |\phi|}B^\vee(\phi)(i_{f\cup g}x)-(-1)^{|[f,g]| \cdot |\phi|} \phi(i_f B(i_gx))\\
&\, \, \, + (-1)^{|[f,g]| \cdot |\phi|+(|f|+1)(|g|+1)}  \phi(i_g B(i_fx)) \\
& \, \, \, + (-1)^{|[f,g]|.|\phi|+|g|} \phi(i_{f\cup g}B(x))\\
&=(-1)^{|[f,g]| \cdot |\phi|+|f|+|\phi| +(|\phi|+1)|f\cup g|}i_{f\cup g}(B^\vee(\phi))(x)\\
&\, \, \, -(-1)^{|[f,g]||\phi|+|g|(|\phi|+|f|+1)+|f|+|\phi|+|f|\cdot |\phi|} i_g( B^\vee(i_f\phi))(x)\\
&\!Ê\! \!  +(-1)^{|[f,g]|\cdot |\phi|+(|f|+1)(|g|+1)+|f|(|\phi|+|g|+1)+|g|+|\phi|+ |g|\cdot |\phi| }  i_f( B^\vee(i_g\phi))(x)\\
&\, \, \, +(-1)^{|[f,g]|\cdot |\phi|+|g|+(|\phi|+1)|f\cup g|+|\phi|} i_{f\cup g} (B^\vee \phi)(x)\\
&=(-1)^{|g|}i_{f\cup g}(B^\vee(\phi))(x)+(-1)^{(|f|+1)(|g|+1)}i_gB^\vee(i_f \phi)(x)\\
&\, \, \,  -i_fB^\vee(i_g \phi)(x) +(-1)^{|f|} B^{\vee}((f\cup g)\cdot \phi)(x).
 \end{split}
\end{equation*}
This proves the statement.
\end{proof}

Now let us suppose that we have an equivalence $A \simeq A^\vee[d]$ in the derived category of $A$-bimodules.
This property provides us with an isomorphism $D: HH^*(A,A^\vee)\rightarrow HH^{*+d}(A,A)$ which allows us to transfer the Connes operator on $HH^*(A,A^\vee)$ to $HH^*(A,A)$,

$$
\Delta := D \circ B^\vee \circ D^{-1}.
$$

\begin{lemma} Let $A$ be a DG algebra with an equivalence $A\simeq {A^\vee}[d]$ in the derived category of $A$-bimodules. Then the induced isomorphism 
$$
D: HH^*(A,A^\vee)\rightarrow HH^{*+d}(A,A)
$$ 
is an isomorphism of $HH^*(A,A)$-modules, i.e. for all $f\in HH^*(A,A)$ and $\phi\in HH^*(A,A^\vee)$ we have
\begin{equation}
D(i_f(\phi))= f\cup D(\phi).
\end{equation}
\end{lemma}
\begin{proof}
The proof is identical to the proof of Lemma~\ref{lem0}. One uses a
resolution by semi-free modules in the category of $A$-bimodules and adapts diagram (\ref{diagcom1}) to the case of $C^*(A,A^\vee)$, the dual theory of $C_*(A,A)$.
\end{proof}

\begin{corollary}\label{coropda} Let $A$ be a DG algebra with an equivalence $A\simeq {A^\vee}[d]$ in the derived category of $A$-bimodules. Then for $f,g\in HH^*(A,A)$ and $\phi \in HH^*(A,A^\vee)$  we have
\begin{equation}
\begin{split}
[f,g]\cup D(\phi)&= (-1)^{|f|} \Delta((f\cup g)\cdot D\phi)-f\cdot  \Delta (g\cdot D\phi)\\
& \, \, \, + (-1)^{(|f|+1)(|g|+1)}g\cdot   \Delta(f\cdot D\phi) +(-1)^{|g|} (f\cup g)\cdot  DB^\vee(\phi),
\end{split}
\end{equation}
where $D:HH^*(A,A^\vee)\rightarrow HH^*(A,A)$ is the isomorphism induced by the derived equivalence.

\end{corollary}
\begin{proof}
This is a  consequence of  Lemma~\ref{lem2}, the proof being similar to that of Corollary~\ref{cor1}.
\end{proof}

\begin{definition}\label{PDPDA}
Let $A$ be a differential graded algebra such that $A$ is equivalent to $A^\vee[d]$ in the derived category of $A$-bimodules.
This means that there is a quasi-isomorphism of $A$-bimodules $\psi: P\rightarrow A^\vee[d]$ where $P$ is a semi-free resolution of $A$.
Then $\psi$ is a cocycle in $\Hom_{A^{e}}(P,A^\vee)$ for which $-\cap [\psi]: HH^*(A,A)\rightarrow HH^*(A,A^\vee)$ is an isomorphism. Under this assumption, $A$ is said to be a \emph{derived Poincar\'e duality algebra} (DPD for short) of dimension $d \in \Z$ if $B^\vee([\psi])=0$.

\end{definition}

\begin{remark} Given a cocycle $\psi \in C^d(A,A^\vee)$, it is rather easy to check when $-\cap [\psi]: HH^*(A,A)\rightarrow HH^{*+d}(A,A^\vee)$ is an isomorphism\footnote{Intuitively, one should think of $HH^*(A,A)$ as the homology of the free loop space of some space, which includes a copy of the homology of the underlying space by the inclusion of constant loops. This condition means that one has to check that the restriction of the cap product to the constants loops corresponds to the Poincar\'e duality of the underlying manifold.}: one only has to check that $-\cap [\phi]: H^*(A)\hookrightarrow HH^*(A,A)\overset{\cap [\phi]}{\hookrightarrow} H^*(A^\vee) $ is an isomorphism (see~\cite{Me}, Proposition 11).
\end{remark}

Two immediate consequences of the previous lemma are the following theorems.

\begin{theorem}
 For a DPD algebra $A$,  $(HH^*(A,A), \cup, \Delta)$ is a BV-algebra, i.e.
\begin{equation}\label{BV-cond}
[f,g]=(-1)^{|f|}\Delta(f\cup g)-(-1)^{|f|}\Delta(f)\cup g-f\cup \Delta(g).
\end{equation}
\end{theorem}
\begin{proof}
Suppose that the derived equivalence $A^\vee[d] \simeq A$ is realized by a quasi-isomorphism $\psi: P\rightarrow A^\vee$,
where $\epsilon: P\rightarrow A$ is a semi-free resolution of $A$.
\begin{equation}
\xymatrix{P\ar[r]^\psi & A^\vee \\ P\ar[r]^{id} & P\ar[u]^\psi  \ar[d]^\epsilon\\ P\ar[r]^\epsilon & A}
\end{equation}
One can then use $\Hom_{A^{e}}(P,A^\vee)$ to compute $HH^*(A,A^\vee)$, and similarly use $\Hom_{A^{e}}(P,A)$ or $\Hom_{A^{e}}(P,P)$ to compute the cohomology $HH^*(A,A)$. Let $ D= (-\cap [\psi])^{-1}: HH^*(A,A^\vee)\rightarrow HH^*(A,A)$
be the isomorphism induced by the derived equivalence. 

Then the cohomology class represented by $id\in \Hom_{A^{e}}(P,P)$ corresponds to $1\in HH^*(A,A)$  using $\epsilon_* : \Hom_{A^{e}}(P,P) \rightarrow \Hom_{A^{e}}(P,A) $,  and to $\psi$ by the map $$\psi_*:\Hom_{A^{e}}(P,P)  \rightarrow \Hom_{A^{e}}(P,A^\vee).$$
Therefore, $D([\psi])= 1\in HH^*(A,A)$ where $D=\epsilon_*\circ \psi_*^{-1}$.  Now take $\phi=[\psi]$ in the statement of Corollary~\ref{coropda}.
\end{proof}

A similar theorem can be proved under a slightly different assumption.

\begin{theorem}(Menichi~\cite{Me}) Let $A$ be a differential graded algebra equipped with a quasi-isomorphism $m: A\rightarrow A^\vee [d]$. Then $HH^*(A,A)$ has a BV-algebra structure extending its natural Gerstenhaber algebra structure. The BV-operator is $\Delta= DB^\vee D^{-1} $ where
$D: HH^*(A,A^\vee)\rightarrow HH^*(A,A)$ is the isomorphism induced by $m$.
\end{theorem}

\begin{proof}
The proof is very similar to that of the previous theorem. For simplicity we take the two-sided bar resolution $ \epsilon: B(A,A,A) \rightarrow A $, where $ \epsilon: A\otimes \kk \otimes A\subset B(A,A,A) \rightarrow A$ is given by the multiplication of $A$. Then $\psi= m\circ \epsilon: B(A,A,A)\rightarrow A^\vee$ is a quasi-isomorphism.
Let $[\psi] \in HH^*(A,A^\vee)$ be the class represented by $\psi$, and $D=(-\cap  [\psi ])^{-1}: HH^*(A,A^\vee)\rightarrow HH^*(A,A)$ be the inverse of the isomorphism induced by $\psi$. Similarly to the proof of the previous theorem, we have $D([\psi])=1\in HH^*(A,A)$. It only remains to prove that $DB^\vee([\psi])=0$. For that we compute $B^\vee([\psi])=[B^\vee (\psi)]$. Note that we have a $\kk$-module isomorphism $C^*(A,A^\vee)=\Hom_{A^{e}}(B(A,A,A), A^\vee)\simeq \Hom(A\otimes T(s\bar{A}), \kk)= (A\otimes T(s\bar{A}))^\vee$. The image of the Connes operator $B:A\otimes T(s\bar{A})\rightarrow A\otimes T(s\bar{A})$ is included in $A\otimes T(s\bar{A})^+$. Since $\epsilon_{|A \otimes T(s\bar{A})^+\otimes A}=0$, we have $m\circ B=\psi\circ\epsilon \circ B=0$.
\end{proof}

\section{Symmetric open Frobenius algebras and (co)BV-structure of Hochschild homology}
\label{sec:Frobenius}

In order to get the correct definitions  we need to fix a few conventions.  

Let $A$ be a differential graded $\kk$-algebra. The $\kk$-dual module $A^\vee=\hom_{\kk}(A,\kk)$ is negatively graded, i.e.
$A^\vee_{-i}= \Hom_{\kk}(A_i,\kk)$ is equipped with the differential $d^\vee$  defined by $d_{A^\vee}(\alpha)(x)=-(-1)^{|x|}\alpha(d_A(x))=(-1)^{|\alpha|}\alpha(d_A(x))$, $\alpha \in A^\vee$, which is also of degree one.  Our choice of sign makes the evaluation map $ev: A\otimes A^\vee \to \kk$ a chain map of degree zero. 
We apply the same rule for a general $A$-bimodule $M$. That is  $M^\vee=\hom_{\kk}(M,\kk)$ is equipped with the differential $d_{M^\vee}\phi=(-1)^{|\phi|}\phi \circ d_M$.

There are natural left and right $A$-module structures on $A^\vee$ given by $x\cdot\alpha : y\mapsto \alpha(yx)$ and $\alpha\cdot x:y\mapsto (-1)^{|x|}\alpha (xy)$. The maps $(x,\alpha)\mapsto x\cdot \alpha$ and  $(\alpha,x)\mapsto \alpha \cdot x$  are chain maps. Similarly for a general $A$-bimodule $M$,   $M^\vee$ is equipped with the $A$-bimodule structure $(x\cdot \alpha)(y):=\alpha(yx)$ and  $(\alpha\cdot x)(y):=(-1)^{|x|}\alpha (xy)$, where $\alpha \in M^\vee$.

For each pair consisting of a right $A$-module $M$ and of a left $A$-module $ N$, note that $M\otimes N$ is an $A$-bimodule, and so is $(M\otimes N)^\vee$.
There is a natural inclusion of $A$-bimodules $i_{N,M}:N^\vee\otimes M^\vee \hookrightarrow (M\otimes N)^\vee  $ given by $\phi_1\otimes \phi_2\mapsto (-1)^{|\phi_1||\phi_2|}\phi_2\otimes \phi_1$. 

We use the following sign rule for  the tensor product of $f \in\Hom_{\kk}(A^{\otimes p} ,M)$ and   $g \in \Hom_{\kk}(A^{\otimes q} ,N)$:
\begin{equation} \label{sign1}
(f\otimes g)(a_1\otimes \cdots \otimes a_{p+q})= (-1)^{|g|(|a_1|+\cdots +|a_p|)}f(a_1\otimes \cdots \otimes a_p)\otimes  g(a_{p+1}\otimes\cdots \otimes a_{p+q}).
\end{equation}

The shift of the degree by $m$  is denoted by $s_{m}:A\to A$ and $\deg (s_{m}(a))= \deg (a) + m$, or in other words $A[m]_k=A_{k-m}$.

Using the shift operation one can  pullback other operations, for instance a product $\mu: A\otimes A\to A$  is pulled back to 
\begin{equation} \label{sign3}
\mu_{m} :=s_{m}^*(\mu)= s_{m} \circ \mu \circ (s_{m}^{-1} \otimes s_m^{-1}),
\end{equation}
or more explicitly $\mu_m(s_{m}(a),s_{m}(b))=(-1)^{m|a|} s_{m}\mu(a,b)$.

So if $\mu$ is of degree  $m$ then $\mu_m$ is of degree zero and  $\mu_m$ is associative (of degree zero) if $\mu$ is associative (of degree $m$), i.e.
$$
\mu(\mu(a,b)c)=(-1)^{m|a|}\mu(a, \mu(b,c)).
$$

%Similarly, the coassociativity rule for a coproduct $\delta$ of degree $m$ is obtained by writing down the usual coassociativity rule for the degree zero coproduct $\delta'= (s_{-m}\otimes s_{-m} )\delta s^{-1}_{-m}$, which translates to 
%$$
%(\delta\otimes 1)\delta=(-1)^m(1\otimes \delta)\delta.
%$$
Similarly, the coassociativity  and cocommutativity conditions for a coproduct $\delta$ of degree $m$ are obtained by writing down the usual coassociativity and cocommutativity conditions for the degree zero coproduct $\delta'= (s_{-m}\otimes s_{-m} )\delta s^{-1}_{-m}$. These translate into 
$$
(\delta\otimes \id)\delta=(-1)^m(\id\otimes \delta)\delta
$$
and
$$
\tau \circ \delta=(-1)^m\delta,
$$
where $\tau:A\otimes A\to A\otimes A$ is  given by  $\tau(x\otimes y)=(-1)^{|x||y|}y\otimes x$.

Using the same argument the equations defining a degree $m$  (right and left) counit  for $\delta$ are  $\id= M (\eta \otimes \id)\delta'= M  (\id \otimes \eta)\delta'$, where $M$ stands for either of the natural isomorphisms
$ A\otimes _\kk\kk \simeq A$ and $ \kk \otimes _\kk A \simeq A$. 
This explicitly amounts to requiring that the identities 
$$
x=\sum_{(x)}  (-1)^{m|x'|}\eta(x')x''=\sum_{(x)}  (-1)^{m|x'|}\eta(x'')x'
$$  
be satisfied for any $x$. Here we use a simplified form of Sweedler's notation for the coproduct, in which we write $ \delta x=\sum_{(x)} x'\otimes x''$ instead of $ \delta x=\sum_{i} x'_i\otimes x''_i$ and interpret $(x)$ as the indexing set for the variable $i$.

\begin{definition}(DG open Frobenius algebra)\label{defopfrob}.  A \textit{differential graded open Frobenius $\kk$-algebra} of degree $m$ is a triple $(A,\cdot,\delta)$  such that:
\begin{enumerate}
\item $(A,\cdot)$ is a unital differential graded associative algebra whose product has degree zero,
\item $(A, \delta)$ is a differential graded coassociative coalgebra  of degree $m$. That $\delta$ is a chain map of degree $m$ means that we have $\delta d=(-1)^m(d\otimes \id+\id\otimes d)\delta$, while coassociativity means as above $(\delta\otimes \id)\delta=(-1)^m(\id\otimes \delta)\delta$.

\item $\delta : A\rightarrow A\otimes A$ is a right and left differential $A$-module map. Using the simplified Sweedler notation this reads
$$
\sum_{(xy)} (xy)'\otimes (xy)''=\sum_{(y)}(-1)^{m|x|} xy'\otimes y''= \sum_{(x)}  x'\otimes x''y.
$$
\end{enumerate}
Note that we must have $\deg x' +\deg x''=\deg x + m$ since the coproduct is assumed to have degree $m$.

We have in particular 
$$
\sum_{(dx)} (dx)'\otimes(dx)'' =(-1)^m  \sum_{(x)} (dx'\otimes x''+(-1)^{|x'|} x'\otimes dx'')
$$
and
$$
\sum (x')'\otimes (x')''\otimes x''=\sum (-1)^{m|x'|+m}x'\otimes (x'')'\otimes (x'')''.
$$
We shall say $(A,\cdot,\delta)$ is \emph{symmetric} if 
$ \sum_{(1)} 1'\otimes 1''= \sum_{(1)}(-1)^{|1'||1''|+m} 1''\otimes 1'$. 
This can be interpreted as a weak version of cocommutativity. In the case of closed 
Frobenius algebras it is a direct consequence of the inner product being symmetric (see Proposition~\ref{pro-frob} below).
\end{definition}

We recall that a closed (DG) Frobenius algebra is a finite dimensional unital associative differential graded $\kk$-algebra $A=\oplus_{i\geq 0} A_i$ equipped with a \emph{symmetric inner product}  $\langle -,- \rangle$ 
such that the map $ \alpha: x\mapsto (y\mapsto  \alpha_x(y):=(-1)^{|x|+m} \langle x, y \rangle=(-1)^{|y|} \langle x, y \rangle)  $ from $A$ to $A^\vee$ is a degree $m$ isomorphism of differential graded $A$-bimodules.  
The condition that the inner product is symmetric means that 
$$ \langle x,y \rangle=(-1)^{|x||y|}\langle y,x \rangle=(-1)^{|x|(m-|x|)}\langle y,x\rangle.$$

Since $\alpha$ is of degree $m$, the explicit condition that characterizes $A$-bi-equivariance takes a slightly more complicated form which involves the degree. Let us spell this out since it is important in order to get the signs right. 
Let $L:A\otimes A^\vee\to A^\vee$ and $R:A^\vee\otimes A\to A^\vee$ be respectively the left and right  action of $A$ on  $A^\vee$.
We will use the same notation for the action of $A$ on itself. That $\alpha$ is $A$-bi-equivariant means that
$$
L\circ (\id\otimes \alpha)=\alpha \circ L
$$
and 
$$
R\circ (\alpha \otimes \id)=\alpha \circ R.
$$
By evaluating on elements $x,y\in A$ we obtain $(-1)^{m|y|}y\cdot \alpha_x=\alpha_{yx}$ and $  \alpha_x\cdot y  =\alpha_{xy}$.  

Note also that the definition of closed Frobenius algebras implies that
$$
\langle xy,z \rangle=\langle x,yz\rangle
$$
and 
 \begin{equation}\label{eqclosed}
  \langle dx,y \rangle=-(-1)^{|x|}\langle x,dy \rangle.
  \end{equation}
 In fact, $\alpha$ being a $A$-bimodule map implies that the inner product is symmetric. 
  
We can now define a coproduct $\delta: A\to A\otimes A$ by requiring the diagram

\begin{equation}\label{diagfrob}
\xymatrix
{ A\ar[r]^{\alpha}\ar[dd]^-{\delta} & A^\vee \ar[rr]^-{\text{dual of the product}} &&(A \otimes A)^\vee \\
&  A^\vee\otimes A^\vee \ar[urr]_-{i_{A,A}}&  &\\
A\otimes A \ar[ur]_-{\alpha\otimes \alpha}}
\end{equation} 
to be commutative. Note that in the diagram above the maps $\alpha$, $\alpha\otimes\alpha$ and  $i_{A,A}$ are isomorphisms because $\dim A <\infty$, therefore $\delta$  exists and is unique because of  the non-degeneracy of the inner product. The coproduct $\delta x=\sum_{(x)}x'\otimes x''$  is characterized by the identity

 \begin{equation}\label{eqclosedco}
\langle x,ab \rangle =  \sum_{(x)} (-1)^{m |x'|} \langle x'',a \rangle \langle x',b \rangle. \end{equation}
Since the inner product has  degree $m$, we  obtain
 \begin{equation}
\langle x,ab \rangle =(-1)^{m|b|+m}  \sum_{(x)}  \langle x'',a \rangle \langle x',b \rangle, \end{equation}
which in the special case  $x=1$ reads
 \begin{equation}
\langle a,b \rangle=  \langle 1,ab \rangle =(-1)^{m|b|+m}  \sum_{(1)}  \langle 1'',a \rangle \langle 1',b \rangle. 
\end{equation}

The coproduct $\delta$  is coassociative of degree $m$ --- this is a good exercise for the reader --- and satisfies condition (3) of Definition \ref{defopfrob} because all the other maps in the diagram (\ref{diagfrob}) are morphisms of $A$-bimodules.
Let us also check this last fact directly. We have
\begin{eqnarray*}
\sum(-1)^{m|(xy)'|} \langle  (xy)'',a \rangle \langle (xy)',b \rangle 
& = & \langle xy,ab \rangle \\
& = & \langle x,yab \rangle \\
& = & \sum (-1)^{m |x'|} \langle x'',ya \rangle \langle x',b \rangle\\
& = & \sum (-1)^{m|x'|} \langle x''y,a \rangle \langle x',b \rangle.
\end{eqnarray*}
Together with the non-degeneracy of the inner product this implies
$$
\sum_{(xy)} (xy)'\otimes (xy)''=\sum_{(y)}x'\otimes x'' y.
$$
Similarly 
\begin{eqnarray*}
\sum(-1)^{m|(xy)'|} \langle  (xy)'',a \rangle \langle (xy)',b \rangle 
& = & \langle xy,ab \rangle \\
& = & (-1)^{|x|(m-|x|)}\langle y,abx \rangle \\
& = & (-1)^{|x|(m-|x|)}\sum (-1)^{m |y'|} \langle y'',a \rangle \langle y',b x\rangle\\
& = & \sum (-1)^{m|y'|} \langle y'',a \rangle \langle xy',b \rangle\\
& = & \sum (-1)^{m|x|}(-1)^{m|xy'|} \langle y'',a \rangle \langle xy',b \rangle,
\end{eqnarray*}
so that
$$
\sum_{(xy)} (xy)'\otimes (xy)''=\sum_{(y)}(-1)^{m|x|}xy'\otimes y''.
$$
In other words, a closed Frobenius algebra over a field is also an open Frobenius algebra. 

By replacing  $b=1$ in (\ref{eqclosedco}), we obtain
 \begin{eqnarray*}
\langle x,a \rangle & = & \langle x,a1  \rangle\\
& = & \sum_{(x)}(-1)^{m|x'|}  \langle x'',a \rangle \langle x',1 \rangle\\
& = & \sum_{(x)} (-1)^{m|x'|}  \langle  \langle x',1 \rangle  x'',a \rangle,
\end{eqnarray*}
 which implies 
  \begin{equation}\label{eqclosedco3}
 x= \sum_{(x)} (-1)^{m|x'|} \langle x',1 \rangle  x''.
 \end{equation}
Similarly by taking $a=1$ we obtain
 \begin{equation*}
\langle x,b \rangle  =\langle \sum (-1)^{m|x'|} \langle x'',1 \rangle  x',b \rangle.
\end{equation*}
The non-degeneracy of the inner product implies that
 \begin{equation}\label{eqclosedco2}
 x=\sum_{(x)}  (-1)^{m|x'|}  \langle x'',1 \rangle  x'.
  \end{equation}
In other words $\eta(x)=\langle x,1\rangle$ is a counit, i.e. 
\begin{equation}\label{eq-cou}
x=\sum_{(x)}  (-1)^{m|x'|}\eta(x')x''=\sum_{(x)}  (-1)^{m|x'|}\eta(x'')x'.
\end{equation}

\bigskip

\noindent \textbf{Example 1.} An important example of symmetric open Frobenius algebra is the cohomology with compact support  $H_c^*(M)$ of an oriented $n$-dimensional manifold $M$  (not necessarily closed).  
Note that $H_c^*(M)$  is naturally equipped with the usual cup product and, using the Poincar\'e duality isomorphism (see~\cite[Theorem 3.35]{Hatcherbook})
$$
H^*_c(M) \simeq H_{n-*}(M),
$$
one can transfer the natural coproduct $\Delta$ from $H_*(M)$ to  $H_c^*(M)$.   Then the triple $(H^*_c(M) , \cup, \Delta)$ is a symmetric open Frobenius algebra whose differential is identically zero. The Frobenius compatibility condition is satisfied because the Poincar\'e duality isomorphism above is a map of $H_c^*(M)$-modules.  Note that this open Frobenius structure is only natural with respect to proper maps.
% or inclusion (of open sets).
 
 If $M$ is closed then $H^*(M)=H^*_c(M)$ is indeed a closed Frobenius algebra since $H^*(M)$ has a counit given by $\int :H^*(M)\to \Z$, the evaluation on the fundamental class of $M$, while Poincar\'e duality is given by capping with the fundamental class.  The non-degenerate inner product is defined by $\langle x,y\rangle:= \int_{[M]} x\cup y$. Over  the rationals it is possible to lift this Frobenius algebra structure to the level of cochains. By a result of Lambrechts and Stanley~\cite{LamStan} there is a connected finite dimensional commutative DG algebra $A$ which is quasi-isomorphic to $C^*(M)$, the cochains on a given $n$-dimensional manifold $M$,  equipped with a bimodule isomorphism $A\rightarrow A^\vee $ inducing the Poincar\'e duality isomorphism $H^*(M)\rightarrow H_{n-*}(M)$.
 
\bigskip

 \noindent \textbf{Example 2.} 
 It is known that the homology of the free loop space of a closed oriented manifold is an open Frobenius algebra\cite{CG}. Similarly the Hochschild cohomology $HH^*(A)$ of a closed Frobenius  algebra is  an open Frobenius algebra \cite{TZ}.

\begin{proposition}\label{pro-frob}
\begin{numlist} 
\item A closed Frobenius algebra is a symmetric open Frobenius algebra.
\item A symmetric open Frobenius  algebra $A$ with a counit is finite dimensional and in fact is a closed Frobenius algebra.
\item A symmetric and commutative open Frobenius algebra is cocommutative.
\item Given any element $z$ of a symmetric open Frobenius  algebra $A$, the element $\sum_{(z)}(-1)^{|z''||z'|} z''z'$ belongs to the center of  $A$.
\end{numlist}
\end{proposition}

\begin{proof}
(1) This follows from the characterization (\ref{eqclosedco3}). Indeed, we have
 \begin{eqnarray*}
\lefteqn{(-1)^{m|b|+m}   \sum_{(1)} (-1)^{|1'||1''|+m}  \langle 1',a \rangle \langle 1'',b \rangle}\\
& = & (-1)^{m|b|}    \sum_{(1)} (-1)^{(m-|a|)(m-|b|)}  \langle 1',a \rangle \langle 1'',b \rangle\\
&= & (-1)^{m|b|+m^2-m(|a|+|b|)+|a||b|}    \sum_{(1)}  \langle 1',a \rangle \langle 1'',b \rangle\\
& = & (-1)^{|a||b|+m|a|+m}    \sum_{(1)}  \langle 1',a \rangle \langle 1'',b \rangle \\
& = & 
(-1)^{|a||b|} \langle b,a\rangle \\
& = & \langle a,b\rangle\\
& = & (-1)^{m|b|+m}\sum_{(1)} \langle 1'',a\rangle\langle 1',b\rangle,
\end{eqnarray*}
so that $\sum_{(1)}  (-1)^{|1'||1''|+m} 1''\otimes 1'=\sum_{(1)}   1'\otimes 1''$.

(2)  The inner product is defined by $\langle x,y\rangle=\eta(xy)$ and is clearly invariant. The identity $x= \sum (-1)^{m|x'|}\eta (x'')x'=\sum (-1)^{m|1'|}\eta (1''x)1'$ proves that $A$ is the $\kk$-linear span of the elements $1'$, and in particular it is finite dimensional.

Now we must prove that $\langle-,-\rangle$ is symmetric.  By the identity (\ref{eq-cou}) we can write $$xy=\sum (-1)^{m|x'|}\eta(x'') x' y=\sum (-1)^{m|1'|}\eta(1''x) 1' y,$$ therefore for all $x$ and $y$
\begin{equation*}
\langle x,y\rangle=\sum (-1)^{m|1'|}\eta(1''x) \eta(1' y).
\end{equation*}
Since $A$ is symmetric, we have
\begin{eqnarray*}
\langle x,y\rangle&=& \sum (-1)^{m|1'|}\eta(1''x) \eta(1' y)Ê\\
& = & \sum (-1)^{m|1''|+|1'||1''|+m}\eta(1'x) \eta(1'' y)\\
& = & \sum (-1)^{m (m-|1'|)+ (m-|x|)(m-|y|))+m}\eta(1'x) \eta(1'' y)\\
& = &\sum (-1)^{m|1'|+|x||y|}\eta(1'x) \eta(1'' y)\\
& = & (-1)^{|x||y|} \sum (-1)^{m|1'|}\eta(1'x) \eta(1'' y) \\
& = & (-1)^{|x||y|} \langle y,x\rangle.
\end{eqnarray*}

(3) We check directly the cocommutativity condition: 
\begin{eqnarray*}
\sum x'\otimes x'' & = & \sum (-1)^{m|x|}x 1'\otimes 1 '' \\
& = & \sum (-1)^{m|x|+|1''||1'|+m} x1''\otimes 1' \\
& = & \sum  (-1)^{m|x|+|1''||1'|-|x||1''|+m}  1''x\otimes 1' \\
& = & \sum (-1)^{m|x|+(|x''|-|x|)|x'|-|x|(|x''|-|x|)+m} x''\otimes x' \\
& = & \sum (-1)^{|x'||x''|+|x|(m+|x|-|x'|-|x''|)+m} x''\otimes x' \\
& = & \sum (-1)^{|x'||x''|+m} x''\otimes x'.
\end{eqnarray*}

(4) This is again a direct check:
\begin{eqnarray*}
x\sum_{(z)} (-1)^{|z''||z'|} z''z' & = & \sum_{(z)}(-1)^{|z''||z'|}  xz''z' \\
& = & \sum_{(1)}(-1)^{(|1''|+|z|)|1'|}x 1''z1' \\
& = & \sum_{(1)}(-1)^{(|1'|+|z|)|1''| +|1'||1''|+m}x 1'z1'' \\
& = & \sum_{(1)}(-1)^{|z||1''|+m}x 1'z1''\\
& = & \sum_{(x)}(-1)^{m|x|+|z| |x''|+m} x'zx'' \\
& = & \sum_{(x)}(-1)^{m|x|+|z| (|1''|-|x|)+m}   1'z1'' x \\ 
& = & \sum_{(1)} (-1)^{m|x|+|z| (|1'|-|x|)+|1'||1''|}  1''z1' x \\ 
& = & \sum_{(1)} (-1)^{m|x|+|z| (|z'|-|x|)+|z'|(|z''|-|z|)}  z''z' x \\
& = & (-1)^{(m+|z|)|x|}(\sum_{(z)}(-1)^{|z''||z'|} z''z')x.
\end{eqnarray*}
We recall that  $\sum_{(z)}(-1)^{|z''||z'|} z''z'$ is of degree $m+|z|$.
\end{proof}

\begin{remark}
We can relax the finite dimensionality condition from the definition of closed Frobenius algebras and only require that the map $\alpha:A\mapsto A^\vee$ be a quasi-isomorphism. Of course such an algebra $A$ may no longer be an open Frobenius algebra. However $HH^*(A,A)\simeq HH^*(A,A^\vee)[m]$ will still be a BV-algebra by the results of Section~\ref{sec:Poincare}.
\end{remark}

\emph{Starting from this point we do not write the signs anymore and, for the sake of simplicity, statements are given and proved over $\Z_2$ though all the results hold indeed over $\Z$.}

\smallskip

As we know $HH^*(A,A^\vee)$ is already equipped with a candidate for the BV-operator, namely the Connes operator $B^\vee $. We just need a product on  $HH^*(A,A^\vee)$, or ideally  a coproduct on the Hochschild chains $C_*(A,A)$. This is given by
\begin{equation} \label{eq-cup-frob}
\theta (a_0[a_1,\cdots,a_n])=\sum_{(a_0), 1\leq i\leq n} (a_0'' [a_1,\cdots,a_{i-1},a_i ]) \otimes (a_0'  [a_{i+1}, \cdots,a_{n}]).
\end{equation}
Then we can define a product of $f,g \in C^*(A,A^\vee)= \Hom (A \otimes T(s\bar{A}), \kk)$ by
$$
(f \ast g)(x):= \mu(f\otimes g )\theta(x),
$$
where $\mu: \kk\otimes \kk \rightarrow \kk$  is the multiplication. More explicitly
$$(f \ast g)(a_0 [a_1,\cdots,a_n])= \sum_{(a_0), 1\leq i< n}  f(a''_0 [a_1,\cdots,a_{i-1}, a_i]) g(a'_0  [a_{i+1}, \cdots,a_{n}]).$$ 
Note that  this coproduct is of degree $m$ so the expected BV and CoBV structures are defined respectively on  $HH_*(A,A)[m]$ and $HH^*(A,A^\vee)[m]$.

In the case of a closed Frobenius algebra this product corresponds to the standard cup product on $HH^*(A,A)$ using the isomorphism
\begin{equation*}
HH^*(A,A)\simeq HH^*(A,A^\vee)
\end{equation*}
induced by the inner product on $A$.  More explicitly, we identify  $A$ with $A^\vee$
using the map $ a\mapsto  (a^\vee(x):=\langle a,x\rangle )$. Therefore to a cochain
$f \in C^*(A,A)$, $f:A^{\otimes n}\to A$,  corresponds a cochain $\tilde{f}\in C^*(A)= C^*(A,A^\vee)$ given
by $\tilde{f} \in \Hom (A^{\otimes n }, A^\vee)\simeq  \Hom (A^{\otimes (n+1)},\kk)$,
\begin{equation*}
\tilde{f}(a_0,a_1,\cdots, a_n):=\langle  f(a_1,\cdots, a_n),a_0\rangle.
\end{equation*}
Note that the map $f\mapsto \tilde{f}$ is of degree $-m$ and its inverse is given by 
\begin{equation*}
f(a_1,\cdots, a_n):=\sum_{(1)}  \tilde {f}(1'',a_1,\cdots, a_n)1' .
\end{equation*}
For two cochains $f: A^{\otimes p} \to A$ and $g: A^{\otimes q} \to A$  in $C^*(A,A)$ we have

\begin{equation*}
\begin{split}
\widetilde{f \cup g}(a_0, a_1,\cdots a_{p+q})&= \langle ( f \cup g)( a_1,\cdots, a_{p+q}), a_0\rangle\\
& = \langle f(a_1,\cdots, a_{p})g(a_{p+1},\cdots, a_{p+q}),a_0\rangle\\ 
&=\langle f(a_1,\cdots, a_{p}),g(a_{p+1},\cdots, a_{p+q})a_0\rangle\\
&= \sum_{(a_0)} \langle f(a_1,\cdots, a_{p}),\langle g(a_{p+1},\cdots, a_{p+q}),a'_0 \rangle a''_0\rangle\\ &= \sum_{(a_0)} \langle f(a_1,\cdots, a_{p}),\langle g(a_{p+1},\cdots, a_{p+q}),a'_0 \rangle a''_0\rangle\\
&=\sum_{(a_0)} \langle f(a_1,\cdots, a_{p}),a''_0\rangle \langle g(a_{p+1},\cdots, a_{p+q}),a'_0 \rangle\\
&=
(\tilde{f} \ast \tilde{g})(a_0[a_1,\cdots, a_{p+q}]).
\end{split}
\end{equation*}

\begin{theorem} \label{themop-cop} For a symmetric open Frobenius algebra $A$,  $(HH_*(A,A)[m], \theta, B)$ is a coBV-algebra.
As a consequence $(HH^*(A,A^\vee)[m], \theta^\vee, B^\vee)$ is a BV-algebra. In particular, if $A$ is closed Frobenius algebra the natural isomorphism 
\begin{equation}\label{iso-clo-frob}
(CC^*(A,A^\vee)[m], \theta^\vee) \simeq (CC^*(A,A),\cup)
\end{equation}
endows $(HH^*(A,A), \cup)$ with a natural BV-algebra structure whose BV-operator is the image of the Connes operators $B$ under the isomorphism~\eqref{iso-clo-frob}. 
\end{theorem}

This statement recovers Tradler's result in~\cite{Tradler} for closed Frobenius algebras.

\begin{proof}
To prove the theorem we show that $(C_*(A,A), \theta, B)$ is a homotopy coBV-coalgebra.
It is a direct check that $\theta$ is co-associative. Just like in the above, we transfer the homotopy for commutativity (in the case of closed Frobenius algebra) of the cup product as given  in Theorem~\ref{Gersten} to 
$C^*(A,A^\vee)$ and then dualize it. It turns out that the resulting formula only depends on the product and coproduct, so it makes also sense for open Frobenius algebras. 
The homotopy for cocommutativity is given by
\begin{eqnarray}\label{homotpy-cocom}
\lefteqn{h(a_0  [a_1, \cdots ,a_n])}\\
&:=& \sum_{(1), 0\leq i<j\leq n+1} ( a_0 [ a_1 ,\cdots, a_i, 1'',a_{j},\cdots, a_{n}] )\otimes (1'  [ a_{i+1},\cdots, a_{j-1}] ),\nonumber 
\end{eqnarray}
where for $j=n+1$ and $i=0$  the corresponding terms are respectively 
$$
 (a_0 [ a_1 ,\cdots ,a_i, 1''] )\otimes (1' [ a_{i+1},\cdots,a_{n}])
$$
and
$$
 (a_0 [ 1'',a_{j},\cdots, a_{n}] )\otimes (1' [ a_{1},\cdots, a_{j-1}]).
$$

It is a direct check that $ hd+(d\otimes \id+ \id\otimes d)h=\theta +\tau \circ \theta$  where $\tau: C_*(A) \otimes C_*(A)\rightarrow C_*(A) \otimes C_*(A)$ is given by  $\tau (\alpha_1\otimes \alpha_2)= \alpha_2\otimes \alpha_1$.

To prove that the 7-term (coBV) relation holds, we use the Chas-Sullivan~\cite{CS1} idea (see also~\cite{Tradler}) from the case of the free loop space, adapted to the combinatorial (simplicial) situation. First we identify the Gerstenhaber co-bracket explicitly.  Let
$$
S:= h+\tau\circ h.
$$
Once having proved that $S$ is, up to homotopy, the deviation of $B$ from being a coderivation for $\theta$, the $7$-term homotopy coBV relation is equivalent to the homotopy co-Leibniz identity for $S$.

\subsubsection*{Co-Leibniz identity:} The idea of the proof is identical to Lemma 4.6 in~\cite{CS1}. We prove that up to some homotopy we have
\begin{equation}\label{coleib}
(\theta \otimes \id) S=(\id\otimes \tau)(S\otimes \id)\theta+ (\id \otimes S) \theta.
\end{equation}
It is a direct check that
$$
(\id\otimes \tau)(h\otimes \id)\theta + (\id \otimes h)\theta=(\theta \otimes \id)h,
$$
so to prove (\ref{coleib})  we should prove that up to some homotopy
\begin{equation}
(\id\otimes \tau)(\tau h\otimes \id)\theta + (\id\otimes  \tau h)\theta=(\theta \otimes \id)\tau h.
\end{equation}

The homotopy is given by $H: C^*(A)\rightarrow C^*(A)^{\otimes 3} $
\begin{eqnarray*}
\lefteqn{H(a_0 [a_1, \cdots ,a_n])}\\
&=& \sum_{0\leq l< i\leq j<k} \sum_{(1),(1)} ( 1''[a_l,\cdots, a_{i-1}]) \otimes (1''  [a_{j},\cdots, a_{k-1}])\\ 
&&\qquad \qquad \qquad \otimes \, a_0[a_{1},\cdots , a_{l},1',a_i,\cdots , a_{j-1},1',a_k,\cdots, a_n].
\end{eqnarray*}
Note that in the sum above the sequence $a_i\cdots a_{j-1}$ can be empty. The identity
\begin{eqnarray*}
\lefteqn{(d\otimes \id \otimes \id+\id\otimes d \otimes \id+ \id\otimes \id \otimes d)H+Hd}Ê\\
& = & (\id\otimes \tau)(\tau h\otimes \id)\theta + (\id\otimes  \tau h)\theta+(\theta \otimes \id)\tau h 
\end{eqnarray*}
can then be checked directly.

\subsubsection*{Compatibility of $B$ and $S$:} The final step is to prove that 
$$
S= \theta B + (B \otimes \id + \id\otimes B) \theta
$$ 
up to homotopy. For that we prove that $h$ is homotopic to $(\theta B)_2+(B\otimes \id)\theta$ and similarly $\tau h \simeq (\theta B)_1+(\id \otimes B)\theta $, where  $\theta B= (\theta B)_1 +(\theta B)_2$ with 
\begin{eqnarray*}
\lefteqn{(\theta B)_1(a_0 [a_1,\cdots ,a_n])}\\
&=& \sum_{0\leq i\leq j\leq n}  \sum_{(1)}( 1'' [a_i,\cdots  , a_{j}])\otimes (1' [a_{j+1}, \cdots, a_n, a_0, \cdots, a_{i-1}])
\end{eqnarray*}
and
\begin{eqnarray*}
\lefteqn{(\theta B)_2(a_0 [a_1,\cdots , a_n])}\\
&= \sum_{0<i<j\leq n} \sum_{(1)}  (1'' [a_{j},\cdots  ,a_n ,a_0, a_1, \cdots  , a_{i}])\otimes (1'[a_{i+1},\cdots , a_{j-1}]).
\end{eqnarray*}
It can be easily checked that
\begin{equation*}
\begin{split}
H(a_0  [a_1,\cdots , a_n])= \sum_{0 \leq k\leq i < j \leq n+1} \sum_{(1)}   (& 1[ a_{k+1},\cdots, a_{i},1'',a_{j},\cdots, a_{n},a_0,\cdots, a_k] )\\ &\otimes (1' [ a_{k+1},\cdots,a_{j-1}] )
\end{split}
\end{equation*}
is a homotopy between $h$ and $(\theta B)_2+(B \otimes \id)\theta$. In the formulae describing $H$, the sequence $a_j,\cdots, a_{i-1}$ can be empty.

While computing $dH$ we see that the terms corresponding to $k=0$ are exactly equal to
\begin{eqnarray*}
\lefteqn{(B\otimes 1)\theta(a_0  [a_1,\cdots ,a_n])}\\
&=&\sum (1 [a_{j},\cdots, a_i,a'_0,a_1,\cdots, a_{j-1}] ) \otimes  (a''_0 [a_{i+1}, \cdots , a_n]).
\end{eqnarray*}
Similarly one proves that $\tau h\simeq (\theta B)_1+(\id \otimes B)\theta $.
\end{proof}

\begin{remark}
By a theorem of F\'elix and Thomas~\cite{FTBV}, this cup product on $HH^*(A,A^\vee)$ is an algebraic model for the Chas-Sullivan product on $H_*(LM)$, the homology of the free loop space of a closed oriented manifold $M$. Here one must work over a field of characteristic zero and for $A$ one can take the closed (commutative) Frobenius algebra provided by the theorem of Lambrechts and Stanley~\cite{LamStan} which asserts the existence of an algebraic model with Poincar\'e duality for $C^*(M)$.
\end{remark}

\begin{theorem}\label{thm-co-op}
For $A$ a symmetric open Frobenius algebra,  $HH_*(A,A)$ can be naturally equipped with a BV-structure whose BV-operator is  Connes' operator and whose product at chain level, before the shift, is given by 
\begin{equation}
a_0[a_1,\cdots,a_n]\circ b_0[b_1,\cdots,b_m]=\begin{cases} 0, & \text{ if } n>0, \\ \sum_{(a_0)}Êa''_0a'_0b_0[b_1,\cdots, b_m], & \text{ otherwise}. \end{cases}
\end{equation}
\end{theorem}
\begin{proof}

Using Proposition~\ref{pro-frob} (4) it is  easily checked that $\circ$ is a chain map. The product $\circ$ is strictly associative. We only have to check this for $x=a[\quad]$, $y=b[\quad]$ and $z=c[c_1,\cdots,c_p]$. We have
\begin{equation*}
(x \circ y)\circ z= \sum (a''_0a'_0b_0)''(a''_0a'_0b_0)'c_0[c_1,\cdots,c_p]= \sum b''_0a''_0a'_0b'_0c_0[c_1,\cdots,c_p],
\end{equation*}
which by Proposition~\ref{pro-frob} (4) is equal to 
\begin{equation*}
\sum a''_0a'_0b''_0b'_0c_0[c_1,\cdots,c_p] =x\circ(y \circ z).
\end{equation*} 

Next we prove that the product is commutative up to homotopy. Indeed the homotopy for $x=a_0[a_1,\cdots,a_n]$ and $y= b_0[b_1,\cdots,b_m]$ is given by
\begin{equation*}
K(x,y)= %\sum_{(a_0)}(-1)^{(|a'_0|+1)(|a''_0|+|a_1|+\cdots +|a_p| +p)} 
a''_0[ a_1 , \cdots , a_n ,  a_0'b_0 , b_1 , \cdots , b_m].
\end{equation*}
In the particular case when $x=a[\quad]$ and $y=b[\quad ]$ and for the external differential $d_1$ we have  
\begin{eqnarray*}
d_1K(x,y)&= &\sum (a'ba''+a''a'b)\\
&=& \sum a1'b1''+x\circ y\\
&=&\sum a1''b1'+x\circ y\\
&=&\sum ab''b'+x\circ y
\end{eqnarray*}
which, by  Proposition~\ref{pro-frob} (4), is
$$
\sum b''b'a+x\circ y=y\circ x+x\circ y.
$$ 
Since we also have $d_0K(x,y)=K(d_0x,y)+ K(x,d_0y)$ (for all $x$ and $y$), we obtain 
$$
dK(x,y)= K(dx,y)+K(x,dy)+ y\circ x+x\circ y. 
$$

The Gerstenhaber bracket is naturally defined to be 
\begin{equation}\label{eq-gers}
\{x,y\}:=  K(x,y) +K(y,x).
\end{equation}
Next we prove that  the identity
\begin{equation}\label{eq-devia}
\{x,y\}= \Delta(x\circ y)+\Delta x \circ y + x\circ \Delta y  
\end{equation}
holds up to homotopy.  
First note that $\Delta x\circ  y =0$ for all $x$ and $y$. A homotopy between all the remaining three terms is given by $H+ K(\id\otimes \Delta )$, where for $x=a_0[a_1,\cdots,a_n]$ and $y= b_0[b_1,\cdots,b_m]$ we define
\begin{equation*}
\begin{split}
H(x,y)=\sum_{(a_0), 1 \leq k\leq m +1}1[b_k,\cdots, b_m,a_0'',a_1,\cdots a_n,a'_0b_0,b_1,\cdots b_{k-1}].
\end{split}
\end{equation*}

Finally we prove that the Leibniz identity
$$
\{x\circ y, z\}= \{x,z\}\circ y + x\circ \{y, z\}
$$  
holds up to homotopy. Note that each of the expressions $\{x,z\}$ and $\{y, z\}$ is either zero or it belongs to $\oplus_{n>1} A^{\otimes n}$. Therefore the right hand side is always zero up to homotopy, so we need to prove that the left hand side is always homotopic to zero. This is clear if either $x$ or $y$ belongs to $\oplus_{n>1} A^{\otimes n}$.  Using the homotopy version of the identity (\ref{eq-devia}) 
we know that $\{x\circ y, z\}$ is homotopic  to $\Delta (x\circ y \circ z)$, which is homotopic to zero if  $z\in  \oplus_{n>1} A^{\otimes n}$.
So we may assume that  $x=a[\quad]$, $y=b[\quad]$ and $z= c[\quad]$, in which case we have
$$
\Delta (x\circ y \circ z)= \Delta ( (a''ab)'' (a''a'b)' c )=\Delta ( b''a''a'b'c).
$$
Since  $a''a'$ is central, we get that $\Delta (x\circ y \circ z)= 1 [a''a'b''b'c]$, which is homotopic to zero. The homotopy is given by  $H(a,b,c)= 1 [a''a',b''b'c]$, whose boundary is
$$
a''a'[b''b'c] +1 [a''a'b''b'c]+ b''b'cÕ[a''a'].
$$
In fact  $a''a'[b''b'c]=b''b'cÕ[a''a']$ because
\begin{equation*}
\begin{split}
\sum a''a'[b''b'c]&= \sum1''a1'[1''b1'c]=\sum 1'a1'[1''b1''c]= \sum c'ab'[b''c'']\\
&=\sum c1'a1'[1''b1'']= \sum c1'ab'[b''1'']  =\sum c1'ab1'[1''1''] \\
&=\sum c1''ab1'[1''1']=\sum ca''b1'[1''a']=\sum c1''b1'[1''a1']\\
&=\sum cb''b'[a''a']=\sum b''bc'[a''a'].
\end{split}
\end{equation*}
\end{proof}

\section{Symmetric commutative open Frobenius algebras and BV-structure on shifted relative Hochschild homology} \label{sec:Frobenius-shifted}

In this section we exhibit a BV-structure on the relative Hochschild homology of a symmetric commutative open Frobenius algebra. In particular we introduce a product on the shifted relative Hochschild homology  of symmetric commutative Frobenius algebras, whose dual could be an algebraic model for the Chas-Sullivan~\cite{CS2} / Goresky-Hingston~\cite{GorHing} coproduct on $H_*(LM,M)$. One should note that Chas-Sullivan worked with an equivariant version of this product on $H_*^{S^1}(LM, M)$, a construction that generalizes the Turaev co-bracket~\cite{Tura}. 
 
 For a commutative DG-algebra  $A$ the \emph{relative Hochschild chain complex} is defined to be
 \begin{equation}
 \widetilde{C}_*(A)= \oplus_{n\geq 1} A\otimes \bar{A}^{\otimes n}
 \end{equation}
equipped with the Hochschild differential, where $\bar{A}$ is the kernel of the augmentation $A\to \kk$. Since $A$ is commutative, $ \widetilde{C}_*(A)$ is stable under the Hochschild differential and fits into the split short exact sequence of complexes
\begin{equation}
\xymatrix{0 \ar[r] & (A, d_A) \ar[r] & C_*(A,A)\ar[r] &  \widetilde{C}_*(A)\ar[r]& 0.}
\end{equation}
The homology of  $\widetilde{C}_*(A)$ is denoted $\widetilde{HH}_*(A)$ and is called \emph{the relative Hochschild homology of $A$}.

\begin{theorem}\label{thm-BV_hom}  The shifted relative Hochschild homology $\widetilde{HH}_*(A)[m-1]$ of a degree $m$ symmetric commutative open Frobenius algebra $A$ is a BV-algebra, whose BV-operator is the Connes operator and whose product is given by
\begin{equation}
\begin{split}
x\cdot y & = \sum_{(a_0b_0)}(a_0b_0)'[ a_1 , \cdots , a_m ,  (a_0b_0)'' , b_1 , \cdots , b_n]\\
& = \sum_{(a_0)}a'_0[ a_1 , \cdots , a_m ,  a_0''b_0 , b_1 , \cdots , b_n]\\
& = \sum_{(b_0)}a_0b_0'[ a_1 , \cdots , a_m ,  b_0'' , b_1 , \cdots , b_n]\\
\end{split}
\end{equation}
for $x=a_0 [a_1,\cdots ,a_m]$ and $y= b_0 [b_1, \cdots ,b_n] \in  \widetilde{C}_*(A) $.
\end{theorem}
\begin{proof}
%Note that the identities above hold because $A$ is an open Frobenius algebra. The product defined above is a chain map and is strictly associative, but is commutative only up to homotopy. The homotopy is given by
Note that the identities above hold because $A$ is an open Frobenius algebra.
The product defined above is a chain map and is strictly associative because of the commutativity condition (hence the cocommutativity of the coproduct on $A$, 
see Proposition~\ref{pro-frob}). However it is commutative only up to homotopy. 
The homotopy is given by
\begin{equation}
\begin{split}
H_{1} (x, y)=&   \sum_{(a_0b_0)} 
            1[  a_{1},\cdots , a_{n},(a_0b_0)', b_1,\cdots , b_m , (a_0b_0)'']
\\ & + \sum_{i=1}^n \sum_{(a_0b_0)} 
            1[ a_{i+1},\cdots , a_{n} ,(a_0b_0)', b_1 , \cdots, b_m , (a_0b_0)''  , a_1,\cdots ,a_i].
\end{split}
\end{equation}
To prove that the 7-term relation holds, we adapt once again the idea of Chas and Sullivan~\cite{CS1} to a simplicial situation. First we identify the Gerstenhaber bracket directly. Let

\begin{equation}
\begin{split}
x\circ y:= \sum _{i= 0 }^m \sum_{(a_0)}  b_0[ b_1 , \cdots, b_{i} , a_{0}' , a_{1}, \cdots, a_{n}, a_{0}'',  b_{i+1}, \cdots, b_{m}],
\end{split}
\end{equation}
and then define $\{x,y\}:=x\circ y + y\circ x$.  Next we prove that the bracket $\{-,-\}$ is homotopic to the deviation of the BV-operator from being a derivation. For that we decompose $\Delta (x\circ y)$ into two pieces:

\begin{equation*}
\begin{split}
B_{1} (x,y):= \sum_{j=1}^{m}  \sum_{(a_0b_0)}1[b_{j+1}, \cdots,  b_{m}, (a_0b_0)', a_1, \cdots, a_n ,  (a_0b_0)'' , b_1 , \cdots , b_j],
\end{split}
\end{equation*}

\begin{equation*}
\begin{split}
 B_{2} (x,y):=  \sum_{j=1}^{m}  \sum_{(a_0b_0)}1[a_{j+1},\cdots,  a_{n}, (a_0b_0)', b_1, \cdots, b_m ,  (a_0b_0)'' , a_1 , \cdots , a_j],
\end{split}
\end{equation*}
so that  $B=B_1+B_2$.  Then $x\circ y $ is homotopic  $ B_1(x,y)+x\cdot By$. In fact the homotopy is given by
\begin{equation*}
\begin{split}
H_{2}(x,y) &= \!\!\!\!\sum_{0\leq j\leq i\leq m}  \sum_{(a_0)}1 [b_{j+1}, \cdots , b_{i} ,  a_{0}' , a_{1}, \cdots, a_{n}, a_{0}'' , b_{i+1},\cdots , b_m,  b_{0}, \cdots, b_{j}].
\end{split}
\end{equation*}
Similarly for $y\circ x$ and $B_2(x\cdot y) +Bx \cdot y$. Therefore we have proved that on $HH_*(A,A)$ the following identity holds:

$$
\{x,y\}=B(x\cdot y)+Bx\cdot y+ x\cdot By.
$$
Now proving the 7-term relation is equivalent to proving the Leibniz rule for the bracket and for the product, i.e. 
$$
\{x,y\cdot z\}= \{x,y\}\cdot z+y\cdot \{x,z\}.
$$
It is a direct check that $x\circ (y \cdot z)=(x\circ y)\cdot z+ y\cdot(x\circ z)$. On the other hand $(y\cdot z) \circ x$ is homotopic to
$(y\circ x) \cdot z  + y\cdot (z\circ x) $ via the homotopy
\begin{eqnarray*}
\lefteqn{H_{3}(x,y,z)}\\
 &=&\!\!\!\!\sum a_0 [a_1, \cdots , a_i, b_0', b_1 , \cdots , b_n,b_0'', a_{i+1},\cdots , a_j , c_0', c_1, \cdots, c_m, c_0'', a_{j+1}, \cdots ,a_p].
\end{eqnarray*}
Here $z=c_0[c_1, \cdots , c_p]$.  This proves that the Leibniz rule holds up to homotopy.
\end{proof}

\begin{remark} The commutativity assumption is only needed to make $\widetilde{C}_*(A)$ a subcomplex of  $C_*(A)$. For the proof of the previous theorem cocommutativity suffices (exercise).
\end{remark}

\begin{remark} In~\cite{ChenGan} Chen and Gan proved that for an open Frobenius algebra $A$, the reduced coHochschild homology of $A$ seen as a \emph{coalgebra}, is a BV-algebra. They also proved that the reduced Hochschild homology  is a BV- and coBV-algebra. It is necessary to take the reduced Hochschild homology in order to get the coBV-structure.
\end{remark}

\section{Closed Frobenius Algebras: Action of the moduli space of curves via Sullivan chord diagrams}\label{sec:chord}

In this section we extend the operations introduced in Section~\ref{sec:Frobenius} to an action of Sullivan chord diagrams on the Hochschild chains $C_*(A,A)$ [and cochains $C^*(A,A^\vee)$] of a closed Frobenius algebra. The main theorem of this section recovers a theorem in~\cite{TZ}, because the inner product induces an isomorphism  $A\simeq A^\vee$ of $A$-bimodules, so that all structures can be transferred from $HH^*(A,A^\vee)$ to $HH^*(A,A)$. Since we are describing the action on the Hochschild chains there is a difference between our terminology and that of~\cite{TZ}. Here the incoming cycles of a Sullivan chord diagram correspond to outgoing cycles of the same diagram in~\cite{TZ} and~\cite{CG}. Such an action has also been described by  Wahl and Westerland~\cite{WahWest}, who work with integral coefficients and explain how to infer from it an action of the moduli spaces of curves (see Section 2.10 in~\cite{WahWest}).

Recall that a \emph{fat graph} is a graph with a cyclic ordering on the set of edges entering each vertex. 
A \emph{Sullivan chord diagram}~\cite{CS2,CG} is a fat graph that is immersed in the plane and inherits the cyclic ordering at each vertex from the orientation of the plane, and which consists of a finite union of labeled disjoint embedded circles, called \emph{output circles} or \emph{outgoing boundaries}, and of \emph{disjointly immersed} trees whose endpoints land on the output circles. The trees are called \emph{chords} and are thought to have length zero. We assume that each vertex is at least trivalent, therefore there is no vertex on a circle which is not an end of a tree. The graphs don't need to be connected.

Any fat graph, and in particular any Sullivan chord diagram, can be thickened uniquely to an oriented surface with boundary. For a Sullivan chord diagram we require that the cyclic ordering be such that all the output circles are among the boundary components. A \emph{Sullivan chord diagram of type $(g,m,n)$} is a Sullivan chord diagram with $n$ output circles and whose underlying fat graph thickens to a Riemann surface of genus $g$ with $m+n$ boundary components.
The remaining $m$ labeled boundary components are called \emph{input circles} or \emph{ingoing boundaries}. 

We also assume that each incoming circle has a marked point, called \emph{input marked point}, and similarly each outgoing boundary has a marked point, called \emph{outgoing marked point}. As in~\cite {WahWest}, one may think of the output marked point as a leaf, connecting  a tree vertex to the corresponding output circle, but we do not adopt this point of view. We do not  consider marked points  as vertices of the graphs. Also, the marked points and the endpoints of the chords may coincide. The \emph{special points} on the graph are by definition the input and output marked points, together with the endpoints of the chords. 
Note that there is a well-defined cyclic ordering on the special points attached to a chord.

\begin{figure}[htb]
\centering
\includegraphics[scale=0.6]{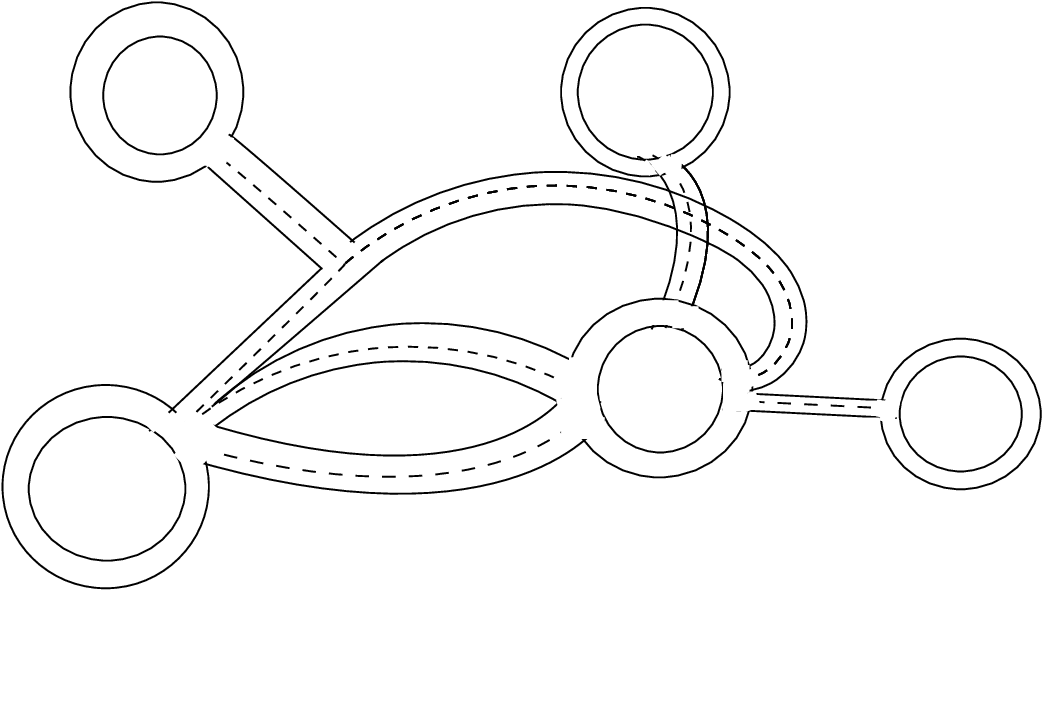}
\caption{}
\label{graph}
\end{figure}

Figure~\ref{graph} displays a chord diagram with 5 outgoing circles and 3 incoming circles. There is an obvious composition rule for two Sullivan chord diagrams if the number of output circles of the first graph equals the number of input circles of the second one. Of course the labeling matters and marked points get identified.  This composition rule makes the space of Sullivan chord diagrams into a PROP (see~\cite{WahWest}  and~\cite{TZ} for more details). Here we don't give the definition of a PROP and we refer the interested reader to~\cite{May} and~\cite{MSS} for more details. 

The combinatorial degree of a diagram of type $(g,m,n)$ is the number of connected components obtained after removing all special points from the output circles, minus $n$. Let $CS_k(g,m,n)$ denote the space of  $(g,m,n)$-diagrams of degree $k$. 
For instance the combinatorial degree of the diagram in Figure~\ref{diag3}  is $1$, which corresponds to the degree of the BV-operator. One makes $\{CS_k(g,m,n)\}_{k\geq 0}$ into a complex using a boundary map which is defined by collapsing an edge (arc) on input circles and considering the induced cyclic ordering. In what follows we describe the action of chord diagrams on chains in $C_*(A,A)$ whose degree is exactly the combinatorial degree of the given diagram. In other words we construct a chain map $CS_k(g,m,n)\rightarrow (\Hom(C_*(A,A)^{\otimes m},C_*(A,A)^{\otimes n}), D:=[d_{Hoch},-])$. Moreover this action is compatible with the composition rule of the diagrams. Said formally, $ C_*(A,A)$ is a differential graded algebra over the differential PROP $\{CS_k(g,m,n)\}_{k\geq 0}$. We won't deal with this last statement.

\subsubsection*{The equivalence relation for graphs and essentially trivalent graphs.} Two  graphs are considered equivalent if one is obtained from the other using one of the following moves:
\begin{itemize}
\item sliding, one at each time,  a vertex on a chord  over edges of the chord. 
\item sliding an input marked point over the chord tree.
\end{itemize} 
By doing so one can easily see that each Sullivan chord diagram is equivalent to a  Sullivan chord diagram for which each vertex is trivalent or has an input marked point, and no input marked point coincides with a chord endpoint.

\subsubsection*{The action of the diagrams.}

Let $\gamma$ be a chord diagram with $m$ input circles and $n$ output circles.  We assume that in $\gamma$ all vertices are trivalent and no input marked point coincides with a chord endpoint (otherwise we replace it with an equivalent trivalent graph as explained above). 

The aim is to associate to $\gamma$ a chain map $C_*(A,A)^{\otimes m} \rightarrow C_*(A,A)^{\otimes  n}$. Let $x_i=a_0^i [a_1^i|\cdots| a_{k_i}^i]$, $1\leq i\leq m$ be Hochschild chains. 

\bigskip

\noindent Step 1) Write down $a_0^i, a_1^i,\cdots, a_{k_i}^i$ on the $i$th input circle by putting first $a_0^i$ on the input marked point and then the rest following the orientation of the circle, on those parts of the $i$th input circle which are not part of the chord tree (at this stage we don't use the output marked point).
We consider all the possible ways of placing $a_1^i,\cdots, a_{k_i}^i$  on the $i$th circle following the rules specified above.

\medskip

\noindent Step 2)  At an output marked point which is not a chord endpoint  or an input marked point  we place a  $1$, otherwise we move to the next step. 

\medskip

\noindent Step 3) On the endpoints of a chord tree  with $r$ endpoints and  no input marked point,  we place following the orientation of the plane $1',1'',\cdots ,1^{(r)} $, where

$$
(\delta\otimes \id ^{(r-2)})\otimes \cdots \otimes (\delta\otimes \id) \delta (1) =\sum_{(1)} 1'\otimes 1''\otimes \cdots \otimes 1^{(r)} \in A^{\otimes r}.
$$

\medskip

\noindent  Step 4) On the endpoints of a chord tree  with $r$ endpoints and which has  $s$  input marked points on its vertices we do the following. We organize the chord tree as a rooted tree whose roots are input marked points.  Now, because of the Frobenius relations, the tree defines an operation $A^{\otimes s}\to A^{\otimes r}$ by using the product and coproduct of $A$. By applying this operation on the element placed on the input marked points (the roots of the tree) we obtain a sum $\sum_i x_i^1\otimes\cdots \otimes x_i^r$.  We decorate the endpoints of the chord tree by   $x_i^r, \cdots , x_i^r$ following the orientation of the plane.

\medskip

\noindent Step 5) For each output circle, starting from its output marked point and following its orientation, we read off all the elements on the outgoing cycle, 
%i.e. the $a_j^i$ left or created after the previous steps, possibly $1$ and the $a^{(k)}$, 
and write them as an element of $C_*(A,A)$. Since the output circles are labeled we therefore obtain a well-defined element of $ C_*(A,A)^{\otimes n}$.
Take the sum over all the labelings/decorations appearing in the previous steps. The result is an element of $C_*(A,A)^{\otimes n}$.

\medskip

Let us illustrate this procedure on some examples. 

\medskip 

The BV-operator~\eqref{B1} corresponds to the diagram in Figure~\ref{diag3}.

\begin{figure}[htb]
\centering
\includegraphics[scale=0.5]{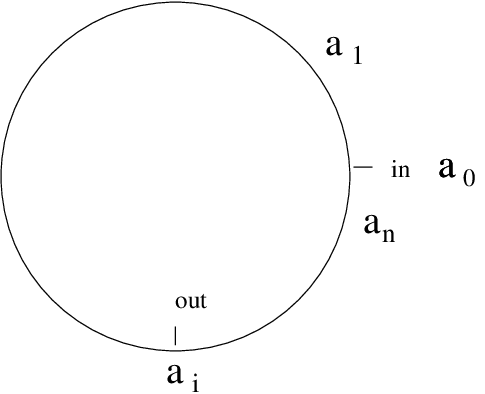}
\caption{BV operator.}
\label{diag3}
\end{figure}

The coproduct (\ref{eq-cup-frob}) 
\begin{equation}
\theta (a_0[a_1,\cdots,a_n])=\sum_{(a_0), 1\leq i\leq n}  (a_0' [a_1,\cdots,a_{i-1},a_i ]) \otimes (a_0''  [a_{i+1}, \cdots,a_{n}])
\end{equation}
corresponds to the diagram in Figure~\ref{Figurenew3}.  The dual of $\theta$ induces a product on $HH^*(A,A^\vee)$ which under the isomorphism $HH^*(A,A^\vee) \simeq HH^*(A,A)$ corresponds  to the cup product on $HH^*(A,A)$ (see Section~\ref{sec:Frobenius}). One should think of the latter as the algebraic model of the Chas-Sullivan product on 
$H_*(LM)$~\cite{CS1}.

\begin{figure}[htb]
\centering
\includegraphics[scale=0.6]{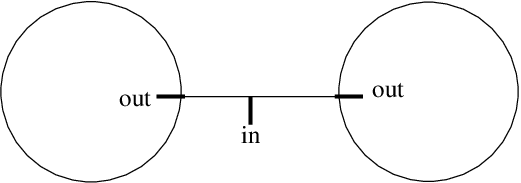}
\caption{String topology coproduct on $HH_*(A,A)$ [the dual of the cup product on cohomology $HH^*(A,A^\vee)\simeq HH^*(A,A)$]. }
\label{Figurenew3}
\end{figure}

The homotopy $h$ defined in (\ref{homotpy-cocom}) for the cocommutativity of $\theta$ corresponds to the diagram in  Figure~\ref{Figurenew4}.

\begin{figure}[h]
\centering
\includegraphics[scale=0.8]{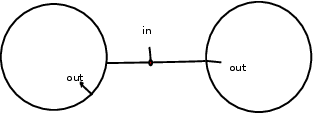}
\caption{The homotopy for cocommutativity of $\theta$.}
\label{Figurenew4}
\end{figure}

The  degree zero coproduct as defined in Cohen-Godin on $H_*(LM)$ is the dual of the following product on $HH_*(A,A)$:
\begin{equation}
(a_0[a_1,\cdots,a_n]) \circ (b_0[b_1,\cdots,b_m])=\begin{cases} 0,  \text{ if } n\geq 1,  \\ a''a'b_0[b_1,\cdots,b_m]  \text{ otherwise.}
\end{cases}
 \end{equation}

The product $\circ$ corresponds to the diagram in Figure~\ref{Figurenew1}, which is equivalent to the essentially trivalent graph in Figure~\ref{Figurenew2}. This is exactly the product introduced in the statement of Theorem~\ref{thm-co-op}.

\begin{figure}[htb]
\centering
\includegraphics[scale=0.6]{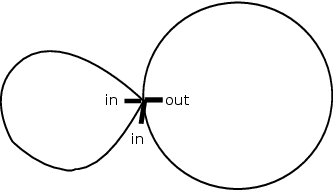}
\caption{The dual of the Cohen-Godin coproduct.}
\label{Figurenew1}
\end{figure}

\begin{figure}[htb]
\centering
 \includegraphics[scale=0.5]{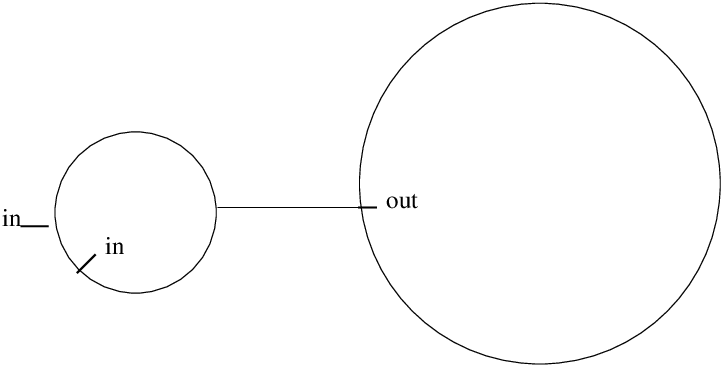}
\caption{Essentially trivalent graph corresponding to Cohen-Godin coproduct.}
\label{Figurenew2}
@\end{figure}

\medskip 

Now it remains to deal with the differentials. This is quite easy to check since collapsing the arcs on the input circles corresponds to the components of the Hochschild differential. The only nontrivial part concerns collapsing the arcs attached to the special points and this follows from the hypothesis that $A$ is an open Frobenius algebra with a counit. This shows that the action associates chain maps to cycles in  $( \{CS_k(g,m,n)\} _{k\geq 0},\partial)$. The homotopies between operations correspond to the action of the boundaries of the corresponding chains in $( \{CS_k(g,m,n)\}_{k\geq 0},\partial)$. We refer the reader to~\cite{TZ} for more details, or to~\cite{WahWest} for a different approach. Now one can explain all the homotopies in the previous section using this language.

The main result of this section can be formulated as follows:

\begin{theorem} For a closed Frobenius (DG) algebra $A$, the Hochschild chain complex $C_*(A,A)$ is an algebra over the PROP of Sullivan chord diagrams. Similarly, the Hochschild cochain complex $C^*(A,A)$ is an algebra over the PROP of Sullivan chord diagrams.
\end{theorem} 

In particular, the action of Sullivan chord diagrams provides to us the compatibility relations between the product and coproduct in order to obtain an open Frobenius algebra (see~\cite{CG} for more details):

\begin{corollary}Let $A$ be a closed Frobenius (DG) algebra. Then $HH_*(A,A)$ and $HH^*(A,A^\vee)$ are open Frobenius algebras.  The product and coproduct for $HH_*(A,A)$  are respectively made explicit in Theorem~\ref{thm-co-op} and Theorem~\ref{themop-cop}.
\end{corollary}
Note that we had already identified the product and coproduct of this open Frobenius algebra structure.

\begin{remark}
As we saw above, the product and coproduct  on $HH_*(A,A)$ require only an open Frobenius algebra structure on $A$.  The results of this section do not prove that the product and coproduct are compatible so that  $HH_*(A,A)$ is an open Frobenius algebra. The reason is that our proof of  the compatibility identities uses the action of some chord diagrams whose actions are defined only if $A$ is a closed Frobenius algebra. Still it could be true  that $HH_*(A,A)$ is an open Frobenius algebra if $A$ is only an open Frobenius algebra, but this needs a direct proof. 
\end{remark}

\section{Cyclic cohomology}\label{sec:cyclic}

In this section we briefly describe some of the structure carried by cyclic homology and negative cyclic homology, which is induced by that of Hochschild cohomology via Connes' long exact sequence. We recall that the \emph{cyclic chain complex}, respectively the \emph{negative cyclic chain complex}  of a  DG-algebra $A$ are
\begin{equation*}
\begin{split}
CC_*(A)&= (C_*(A,A)\otimes_\kk \kk[[u,u^{-1}]/u\kk[[u]], d+ uB),\\
CC_*^-(A)&= (C_*(A,A)\otimes_\kk \kk[[u]], d+ uB).\\
\end{split}
\end{equation*}
Here $u$ is a formal variable of degree $2$,  $d=d_{Hoch}$ and $B$ is the Connes operator 
%whereas $\kk[[u,u^{-1}]$ stands for the Laurent series in $u$ and $u\kk[[u]]$ stands for the submodule $C_*(A,A)\otimes_\kk u\kk[[u]]\subset C_*(A,A)[[u,u^{-1}]$. 
 The \emph{cyclic homology of $A$}\, is denoted $HC_*(A)$ and is the homology of the complex $CC_*(A)  $. The \emph{negative cyclic homology of $A$}Ê\, is denoted $HC^-_*(A)$ and is the homology of $CC_*^-(A)$.  
 
The \emph{cyclic cochain complex}, respectively the \emph{negative cyclic cochain complex} are defined to be:  
\begin{equation*}
\begin{split}
CC^*(A)&= (C^*(A,A^\vee)\otimes_\kk \kk[v], d^\vee+ vB^\vee),\\
CC^*_-(A)&= (C^*(A,A^\vee)\otimes_\kk \kk[v,v^{-1}]]/v\kk[v], d^\vee+ vB^\vee).
\end{split}
\end{equation*}
Here $v$ is a formal variable of degree $-2$. The corresponding cohomology groups are called \emph{cyclic cohomology}, respectively \emph{negative cyclic cohomology}, and are denoted $HC^*(A)$, resp. $HC^*_-(A)$.

\begin{lemma}[\cite{CS1}] Let $(A^*,\cdot,\Delta)$ be a BV-algebra and $L^*$ a graded vector space with a long exact sequence
\begin{equation}\label{Connes}
\xymatrix{\cdots \ar[r]& L^{k+2} \ar[r]  & L^{k} \ar[r]^-m  &  A^{k+1}  \ar[r]^-e & L^{k+1}\ar[r] & L^{k-1} \ar[r]^-m& A^k \ar[r]^-e &\cdots}
\end{equation}
such that $\Delta= m\circ e$. Then
$$
\{a,b\}:=(-1)^{|a|} e(ma\cdot mb)
$$
defines a graded Lie bracket on the graded vector space $L^*$. Moreover $m$ sends the Lie bracket to the opposite of the Gerstenhaber bracket, i.e.
$$
m\{a,b\}=-[ma,mb].
$$

\end{lemma}
\begin{proof}
We have,
\begin{equation}\begin{split}
&\{a,\{b,c\}\}=(-1)^{|a|+|b|}  e (ma \cdot \Delta(mb \cdot mc) )\\
&\{\{a,b \},c\}= (-1)^{|b|}  e (\Delta(ma \cdot mb) \cdot mc)\\
&\{b,\{a,c\}\}=(-1)^{|b|+|a|} e( mb \cdot \Delta(ma \cdot mc)) .
\end{split}
\end{equation}
Then
\begin{eqnarray}
\lefteqn{\{a,\{b,c\}\}- \{\{a,b \},c\}-(-1)^{|a||b|} \{b,\{a,c\}\}}\\
&=& (-1)^{|a|+|b|} e [ma \cdot \Delta(mb \cdot mc) 
+(-1)^{|a|+1}\Delta(ma \cdot mb) \cdot mc\nonumber \\
&& \qquad \qquad \qquad \qquad +(-1)^{|a||b|+1}mb \cdot \Delta(ma \cdot mc)]\nonumber\\
&=&(-1)^{|b|+1}e [\Delta(ma \cdot mb) \cdot mc 
+(-1)^{|a|+1} ma \cdot \Delta(mb \cdot mc)\nonumber \\
&& \qquad \qquad \qquad \qquad +(-1)^{|a|(|b|+1)}mb \cdot \Delta(ma \cdot mc)].\nonumber
\end{eqnarray}

By replacing $a$, $b$, and $c$ in the 7-term relation (\ref{7term}) respectively by $ma$, $mb$ and $mc$, we see that the last line in the above identity is equal to 
$$
(-1)^{|b|+1}e  \Delta (ma\cdot mb\cdot mc)=(-1)^{|b|+1}e m e  (ma\cdot mb\cdot mc)=0
$$ 
because of the exactness of the long exact sequence.
Therefore $ \{a,\{b,c\}\}- \{\{a,b \},c\}-(-1)^{|a|.|b|} \{b,\{a,c\}\}=0$, proving the Jacobi identity.

As for  the second statement,
\begin{equation}\begin{split}
m\{a,b\}&=(-1)^{|a|} me(ma\cdot mb)= (-1)^{|a|} \Delta (ma\cdot mb)\\&=(-1)^{|a|} ( (-1)^{|a|+1}[ma,mb] - \Delta (ma)\cdot mb+ (-1)^{|a|+1}ma\cdot \Delta (mb)]\\
&=-[ma,mb].
\end{split}
\end{equation}
\end{proof}

Using this lemma and Connes' exact sequence for the cyclic cohomology (or homology), 

\begin{equation}\label{Connes2}
\xymatrix{\cdots HC^{k+2}(A) \ar[r]  & HC^{k}(A) \ar[r]^-b  &  HH^{k+1}(A,A^\vee) \ar[r]^-i & HC^{k+1}(A)\cdots}
\end{equation}
we have:

\begin{corollary}
The cyclic cohomology and negative cyclic cohomology of an algebra whose Hochschild cohomology is a BV-algebra, has a natural graded Lie algebra structure given by
$$
\{x,y\}:= i(b(x)\cup b(y)).
$$

\end{corollary}

In fact one can prove something slightly stronger.

\begin{definition}
A \emph{gravity algebra} is a graded vector space $L^* $ equipped with maps
$$
\{\cdot, \cdots, \cdot\}:  L^{\otimes k}\rightarrow L
$$
satisfying the following identities: 
\begin{equation}
\begin{split}
\sum_{i,j}(-1)^{\epsilon_{i,j}}\{\{x_i,&x_j\},x_1, \cdots ,\hat{x}_i,\cdots, \hat{x}_j,\cdots,x_k,y_1,\cdots,y_l\}\\&=\begin{cases}
0, \text{ if } l=0,\\
\{ \{ x_1,\cdots,x_k\},y_1,\cdots,y_l \}, \text{ if } l>0,
\end{cases}
\end{split}
\end{equation}
where $\epsilon_{i,j}=|x_i| (\sum_{k=1}^{i-1}|x_k|) +|x_j| (\sum_{k=1, k\neq i}^{j-1} |x_k|)$.
\end{definition}

It is quite easy to prove that

\begin{proposition}  The cyclic and negative cyclic cohomology of an algebra whose Hochschild cohomology is a BV-algebra are naturally gravity algebras, where the brackets are given by
$$
\{x_1,\cdots,x_k\}:= i(b(x_1)\cup \cdots \cup b(x_k)).
$$
\end{proposition}

The proof is a consequence of the following identity for BV-algebras:
\begin{equation}
\begin{split}
\Delta( x_1\cdots x_n)=\sum \pm \Delta(x_ix_j ) \cdot  x_1\cdots \hat{x}_i \cdots \hat{x}_j \cdots x_n.
\end{split}
\end{equation}
This is a generalized form of the 7-term identity which is rather easy to prove. We refer the reader to~\cite{West} for a more operadic approach on the gravity algebra structure.

\medskip

The Lie bracket on cyclic homology is known in the literature under the name of \emph{string bracket}. For surfaces it was discovered by W. Goldman~\cite{Gold} who studied the symplectic structure of the representation variety of fundamental groups of surfaces, or equivalently the moduli space of flat connections. His motivation came from Teichm\"uller theory.
%in the dynamics of Teichm\"uller theory and Hamiltonian vector fields of  Thurston earthquakes. 
The Goldman bracket was generalized by Chas and Sullivan to manifolds of all dimensions using a purely topological construction. The geometric description of the string bracket given in~\cite{AZ} (and~\cite{ATZ}) generalizes Goldman's computation for surfaces to arbitrary even dimensions using Chen iterated integrals.

\bibliography{Bib-IRMA-Hoch}

\bibliographystyle{amsalpha}

\end{document}